\documentclass[english,11pt,reqno]{amsart}
\usepackage{amsmath,amsthm,amsfonts,amssymb}
\usepackage{graphicx}
\usepackage[ansinew]{inputenc}
\usepackage{a4wide}

\usepackage{hyperref}

\newcommand{\length}{\operatorname{length}}

\newcommand{\leb}{\operatorname{Leb}}

\newcommand{\dist}{\operatorname{dist}}
\newcommand{\diam}{\operatorname{diam}}

\newcommand{\lip}{\operatorname{Lip}}

\def \RR {{\mathbb R}}
\def \ZZ {{\mathbb Z}}
\def \NN {{\mathbb N}}
\def \PP {{\mathbb P}}
\def \TT {{\mathbb T}}
\def \II {{\mathbb I}}
\def \JJ {{\mathbb J}}

\title[SRB measures for partially hyperbolic systems]{SRB measures for partially hyperbolic systems whose central direction is weakly expanding}

\author{José F. Alves}
\address{Jos\'e F. Alves\\ Departamento de Matem\'atica, Faculdade de Ci\^encias, Universidade  do Porto\\
Rua do Campo Alegre 687, 4169-007 Porto, Portugal}
\email{jfalves@fc.up.pt} \urladdr{http://www.fc.up.pt/cmup/jfalves}

\author{Carla L. Dias}
\address{Carla L. Dias\\ Instituto Politécnico de Portalegre, Lugar da Abadessa,
Apartado 148,
7301-901 Portalegre, Portugal}
\email{carlald.dias@gmail.com}

\author{Stefano Luzzatto}
\address{Stefano Luzzatto\\Abdus Salam International Centre for Theoretical Physics, Strada Costiera 11, Trieste, Italy.}
\email{luzzatto@ictp.it}
\urladdr{http://www.ictp.it/$\sim$luzzatto}

\author{Vilton  Pinheiro}
\address{Vilton Pinheiro\\Departamento de Matem\'atica, Universidade Federal da Bahia\\
Av. Ademar de Barros s/n, 40170-110 Salvador, Brazil.}
\email{viltonj@ufba.br}

\begin{document}

\thanks{JFA was partially supported by Funda\c c\~ao Calouste Gulbenkian, CMUP, the European Regional Development Fund through the Programme COMPETE, and FCT under the projects PTDC/MAT/099493/2008, PTDC/MAT/120346/2010 and PEst-C/MAT/UI0144/2011. CLD was supported by FCT}



\begin{abstract}
We consider partially hyperbolic \( C^{1+} \) diffeomorphisms of compact Riemannian manifolds of arbitrary dimension which admit a partially hyperbolic tangent bundle decomposition \( E^s\oplus E^{cu} \). Assuming the existence of a set of positive Lebesgue measure on which \( f \) satisfies a weak nonuniform expansivity assumption in the centre~unstable direction, we prove that there exists at most a finite number of transitive attractors each of which supports an SRB measure. As part of our argument, we prove that each attractor admits a Gibbs-Markov-Young geometric structure with integrable return times. We also characterize in this setting SRB measures which are liftable to Gibbs-Markov-Young structures.
 \end{abstract}

\maketitle


\newcommand{\mcup}{\mbox{$\bigcup$}}
\newcommand{\mcap}{\mbox{$\bigcap$}}

\def \RR {{\mathbb R}}
\def \ZZ {{\mathbb Z}}
\def \NN {{\mathbb N}}
\def \PP {{\mathbb P}}
\def \TT {{\mathbb T}}
\def \II {{\mathbb I}}
\def \JJ {{\mathbb J}}

\def \vare {\varepsilon }

 \def \cf {\mathcal{F}}
 \def \cm {\mathcal{M}}
 \def \cn {\mathcal{N}}
 \def \cq {\mathcal{Q}}
 \def \cp {\mathcal{P}}
 \def \cb {\mathcal{B}}
 \def \cc {\mathcal{C}}
 \def \cs {\mathcal{S}}
 \def \bc {\mathcal{B}}
 \def \hc {\mathcal{C}}

\newcommand{\dem}{\begin{proof}}
\newcommand{\cqd}{\end{proof}}

\newcommand{\qand}{\quad\text{and}\quad}

\newtheorem{theorem}{Theorem}
\newtheorem{corollary}{Corollary}

\newtheorem*{Maintheorem}{Main Theorem}

\newtheorem{maintheorem}{Theorem}
\renewcommand{\themaintheorem}{\Alph{maintheorem}}
\newcommand{\cmt}{\begin{maintheorem}}
\newcommand{\fmt}{\end{maintheorem}}

\newtheorem{maincorollary}[maintheorem]{Corollary}
\renewcommand{\themaintheorem}{\Alph{maintheorem}}
\newcommand{\cmc}{\begin{maincorollary}}
\newcommand{\fmc}{\end{maincorollary}}

\newtheorem{T}{Theorem}[section]
\newcommand{\ct}{\begin{T}}
\newcommand{\ft}{\end{T}}

\newtheorem{Corollary}[T]{Corollary}
\newcommand{\cco}{\begin{Corollary}}
\newcommand{\fco}{\end{Corollary}}

\newtheorem{Proposition}[T]{Proposition}
\newcommand{\cpr}{\begin{Proposition}}
\newcommand{\fpr}{\end{Proposition}}

\newtheorem{Lemma}[T]{Lemma}
\newcommand{\cle}{\begin{Lemma}}
\newcommand{\fle}{\end{Lemma}}

\newtheorem{Sublemma}[T]{Sublemma}
\newcommand{\csle}{\begin{Sublemma}}
\newcommand{\fsle}{\end{Sublemma}}

\theoremstyle{definition}
\newtheorem{Conjecture}{Conjecture}

\newtheorem*{Conjecture*}{Conjecture}

\theoremstyle{remark}

\newtheorem{Remark}[T]{Remark}
\newcommand{\cre}{\begin{Remark}}
\newcommand{\fre}{\end{Remark}}

\newtheorem{privrem}{Private remark}

\newtheorem{Definition}{Definition}
\newcommand{\cd}{\begin{Definition}}
\newcommand{\fd}{\end{Definition}}

\section{Introduction}

A key outstanding question in the theory of Dynamical Systems, which   motivates a large amount of research, is the \emph{Palis conjecture} \cite{P05}, which states that ``typically" dynamical systems on finite dimensional manifolds have a finite number of ``ergodic attractors''.   More specifically,  a Borel probability measure \(  \mu  \) on \(  M  \) is a \emph{physical measure} if there exists a
positive Lebesgue measure set of points $x \in M$ such that for any continuous $\varphi:M \to \RR$ one has
\begin{equation*}
    \lim_{n\rightarrow\infty}\frac{1}{n}\sum_{j=0}^{n-1}\varphi(f^j(x)) \longrightarrow \int\varphi\, d\mu .
\end{equation*}
The set of points \(  \mathcal B(\mu)  \) for which the above convergence holds is  called the \emph{basin} of \(  \mu  \). The Palis conjecture then says that typical dynamical systems admit at least one and at most a finite number of physical measures and that their basins have full Lebesgue measure in \( M \).

The Palis conjecture has been proved in the setting of smooth one-dimensional maps  for  various notions of ``typical'' \cite{Lyu02, AviLyuMel03, KSS07},  but it is still completely open in higher dimensions where the general problem of proving ergodicity or even finiteness of the number of ergodic components of a measure is extremely difficult. In this paper we give a contribution to this area of research by proving that certain natural classes of partially hyperbolic systems admit SRB measures which can have at most a finite number of ergodic components. As part of our argument we also show that these SRB measures are associated to particular geometric structures.  We give the precise definitions below as well as a more detailed discussion and comparison of existing results.

\subsection{SRB measures}
Throughout this paper we
let $M$ be a finite dimensional compact Riemannian manifold and $f:M \rightarrow M$ a diffeomorphism of class \( C^{1+}\), meaning that \( f \) is \( C^1\) with H\"older continuous derivative. We denote by $\leb$ a normalized
volume form on the Borel sets of $M$ that we call  {\em Lebesgue
measure.} Given a submanifold $\gamma\subset M$ we use $\leb_\gamma$
to denote the measure on $\gamma$ induced by the restriction of the
Riemannian structure to~$\gamma$.

In our setting we will work with a particular class of physical measures:
 we say that an ergodic \(  f  \)-invariant probability measure \(  \mu  \) is an
\emph{SRB measure} if it has no zero Lyapunov exponents \(  \mu  \) almost everywhere and the conditional measures on unstable manifolds are absolutely continuous with respect to the Lebesgue measures on these manifolds.
It follows by standard results on the absolute continuity of the stable foliation that SRB measures are a particular form of physical measures \cite{Bow75, Pes77, PS82, You02}.

\subsection{Hyperbolicity}The only general class of dynamical systems in higher dimensions for which the problem of the number of physical measures has been solved is that of  \emph{uniformly hyperbolic} diffeomorphisms, i.e.  systems which admit a continuous invariant tangent bundle decomposition \( TM=E^{s}\oplus E^{u} \) such that the differential map is uniformly contracting on \( E^{s} \) and uniformly expanding on \( E^{u} \). Indeed, it has long been known, since the classical work of Anosov, Smale, Ruelle and Bowen \cite{Bow75, BR75, Rue76, Sin72} in the 1970's, that such systems can admit at most a finite
number of physical measures.

The problem remains wide open for the natural generalisations of uniform hyperbolicity, namely \emph{nonuniform hyperbolicity} where the decomposition  \( TM=E^{s}\oplus E^{u} \) is only measurable and the contraction and expansion estimates are only asymptotic and not uniform \cite{BarPes07, Pes77}, and \emph{partial hyperbolicity} where the tangent bundle decomposition takes the form \( TM=E^{s}\oplus E^{c}\oplus E^{u} \) which is still assumed to be continuous and to admit uniform contraction and expansion estimates in \( E^{s} \) and \( E^{u} \) respectively, but also  includes a ``central'' direction on which very little is assumed \cite{BriPes73, Pes04}, we give a more precise definition below. A vast literature exists concerning the properties of systems satisfying such weak hyperbolicity conditions and several papers address specifically the  existence of SRB measures under various kinds of additional assumptions, most of which include the assumption that the system be \emph{both} partially hyperbolic and nonuniformly hyperbolic (in the central direction),  see \cite{ABV, AlvAraVas07, AlvPin10, BurDolPesPol08, BV00, Cas02, CatHeb01, CowWil05,  Dol00, Dol04, FieNic04, Hat06, HerHerTahUre11, PS82, Vas07, ViaYan13}.

The main purpose of this paper is to relax  the conditions of nonuniform hyperbolicity in the central direction assumed in \cite{ABV}, thus obtaining a significantly more general result which we stress is not contained in any other existing result. This generalization requires a completely different approach  which is more powerful in this setting and probably more suited to potential further extensions.

\subsection{Partial hyperbolicity}
We say that a forward invariant compact set \(  K \subseteq M \)  admits a
\emph{dominated decomposition}
if
there is a continuous \(  Df  \)-invariant splitting $T_K M=E^{cs}\oplus
E^{cu}$  and there exists a Riemannian metric on $M$ and a
  constant $\lambda<1$ such that  for all $x\in K$
 \begin{equation}\label{domination1}
    \|Df |_{E^{cs}_x}\|
\cdot \|Df^{-1} |_{E^{cu}_{f(x)}}\| < \lambda.
\end{equation}
We say that \(  f  \) is
\emph{partially hyperbolic} if moreover we have for all \(  x\in K  \)
\[
\|Df |_{E^{cs}_x} \| < \lambda.
\]
To emphasize the uniform contraction  we shall write \(  E^{cs}=E^{s}  \), so that we have
\[
T_K M=E^{s}\oplus
E^{cu}.
\]
We remark that, more generally, a diffeomorphism with dominated decomposition is said to be partially hyperbolic
 if at least one of the sub-bundles \(  E^{cs}\) or \(E^{cu}  \) admits uniform contraction or uniform expansion,  respectively. In this paper we  are assuming that it is the stable sub-bundle which admits uniform estimates.

\subsection{Nonuniform expansion}
Without  additional assumptions on the centre-unstable bundle \(  E^{cu}  \) it is very difficult to obtain any results at all about the dynamics.
In \cite{ABV},  the existence of physical (SRB) measures was proved under the assumption that
there exists a set \(  H  \) of positive Lebesgue measure
on which \(  f  \) is
\emph{non-uniformly expanding} along~\(  E^{cu}  \):   there exists an \(  \epsilon>0  \)
and some choice of Riemannian metric on \(  M  \) such that for all \(  x\in H  \)
\begin{equation} \label{NUE1}
\limsup_{n\to+\infty} \frac{1}{n}
\sum_{j=1}^{n} \log \|Df^{-1}|_{E^{cu}_{f^j(x)}}\|<-\epsilon.
\end{equation}
In this paper we address a subtle but non-trivial generalization of this result which has remained open for over a decade.
We say that the map  \(  f  \) is   \emph{weakly non-uniformly expanding} along \(  E^{cu}  \)  on a set \(  H  \)
if there exists an \(  \epsilon>0  \) and some choice of Riemannian metric on \(  M  \)  such that for all \(  x\in H  \)
\begin{equation} \label{NUE2}
\liminf_{n\to+\infty} \frac{1}{n}
\sum_{j=1}^{n} \log \|Df^{-1}|_{E^{cu}_{f^j(x)}}\|< - \epsilon.
\end{equation}
We emphasize that the  \(  \liminf  \) condition \eqref{NUE2} implies that the growth  only needs to be verified  on a subsequence of iterates,  in contrast to the limsup condition \eqref{NUE1},  where the condition needs to be verified for all sufficiently large times. The techniques and methods that we use below to deal with this weaker assumptions are completely different from those used in~\cite{ABV}.

\cmt
\label{thA}
Let \( f: M\to M \) be a  \( C^{1+} \)
diffeomorphism,  $K \subseteq M$   a forward invariant compact set on which
\(  f   \) is  partially hyperbolic with splitting \(  T_KM=E^s\oplus E^{cu}\), and $H \subseteq K$ a set with positive Lebesgue measure on which
$f$ is weakly non-uniformly expanding  along \(  E^{cu}\).
Then

\begin{enumerate}
\item
there exist  closed invariant transitive sets \(  \Omega_{1},...,\Omega_{\ell}  \)
such that
  for Lebesgue almost every \(  x\in H  \) we have
  \(  \omega(x)=\Omega_{j}  \) for some  \(  1\leq j\leq \ell \);
\item there exist
 SRB measures \(  \mu_{1},..., \mu_{\ell}  \) supported on the sets \(  \Omega_{1},...,\Omega_{\ell}  \), whose basins have nonempty interior, such that
  for Lebesgue almost every \(  x\in H  \) we have
\( x\in \mathcal B({\mu_{j}})  \) for some  \(  1\leq j\leq \ell \).
\end{enumerate}
\fmt

We remark that our argument works  also if we let \(  \epsilon=0  \) in \eqref{NUE2}, in which case we get a countable number of transitive sets and corresponding SRB measures supported on them.
As the ergodic basins \(\mathcal B(\mu_j)\) have non-empty interior, in the special case in which \(  f  \) is transitive, partially hyperbolic and  weakly nonuniformly expanding along \(  E^{cu}  \) on the whole manifold~\(  M  \), we get the following consequence.

\cmc\label{th:1}
Let \( f: M\to M \) be a  \( C^{1+} \)  transitive partially hyperbolic
diffeomorphism with splitting \(  TM=E^s\oplus E^{cu}\) which is  weakly nonuniformly expanding along  \(  E^{cu}  \) on a subset  of full Lebesgue measure. Then
\(  \omega(x)=M  \) for Lebesgue almost every \(  x  \) and \(  f  \) has a unique SRB measure whose basin has full Lebesgue measure in \(  M  \).
\fmc

It remains an interesting question whether these results hold true under the weaker assumption that the map \(  f  \) has \emph{positive Lyapunov exponents along \(  E^{cu}  \)}: there exists a set \( H \) of positive Lebesgue measure and there exists some \(  \epsilon>0  \) such that
\begin{equation}\label{poslyap}
\limsup_{n\to\infty} \frac 1n \log \|Df^{n}(x)v\| > \epsilon,
\end{equation}
for every \(  x\in H  \) and  every non-zero vector \(  v\in E^{cu}_{x}  \).  If \(  \dim(E^{cu})=1  \) then \eqref{poslyap} is equivalent to \eqref{NUE2} and therefore analogous results as in  Theorem \ref{thA}  and Corollary \ref{th:1}  hold under this weaker assumption.
Observe that condition \eqref{poslyap}, unlike condition \eqref{NUE2}, does not depend on the choice of metric. Moreover, condition \eqref{poslyap} is strictly weaker than \eqref{NUE2} for a given norm; see for instance the example in \cite[Section 4]{AlvAraSau}.

\begin{Conjecture*}
Assume that \(  f  \) is partially hyperbolic and satisfies \eqref{poslyap} along \(  E^{cu}  \) on a set \(  H  \). Then
there exists a Riemannian metric on \(  M  \) such that \(  f  \) satisfies \eqref{NUE2} on \(  H  \).
\end{Conjecture*}

If this conjecture is true, we will immediately obtain the conclusions of Theorem \ref{thA}
 and of Corollary \ref{th:1} under the a-priori weaker condition of positive Lyapunov exponents.

\subsection{GMY structures}

The strategy which we will use to prove the results above, and which is completely different from the
approach in \cite{ABV}, is the construction of certain  geometric structures whose existence is of independent interest. These geometric structures were introduced in \cite{You98} and have been   applied to study the existence and properties of physical measures in certain classes of dynamical systems.  We give here the precise definitions.

An embedded disk \( \gamma\subset M\) is called an
\emph{unstable manifold} if \( d(f^{-n}(x), f^{-n}(y))\to 0\) exponentially fast as \( n \to \infty\) for all \(x,y\in\gamma\); similarly  \( \gamma\subset M\) is called a
\emph{stable manifold} if \( d(f^{n}(x), f^{n}(y))\to 0\) exponentially fast as \( n \to \infty\).
 We say that $\Gamma^u=\{\gamma^u\}$ is a
\emph{continuous family of $C^1$ unstable manifolds} if there is a
compact set~$K^s$,  a unit disk $D^u$ of some $\RR^n$, and a map
$\Phi^u\colon K^s\times D^u\to M$ such that
\begin{itemize}
\item[i)] $\gamma^u=\Phi^u(\{x\}\times D^u)$ is an unstable
manifold;
\item[ii)] $\Phi^u$ maps $K^s\times D^u$ homeomorphically onto its
image; \item[iii)] $x\mapsto \Phi^u\vert(\{x\}\times D^u)$ defines a
continuous map from $K^s$ into $\text{Emb}^1(D^u,M)$.
\end{itemize}
Here   $\text{Emb}^1(D^u,M)$ denotes the space of $C^1$ embeddings from
$D^u$ into $M$.
Continuous families of $C^1$ stable manifolds are defined similarly.

We say that  a set $\Lambda\subset M$ has a \emph{hyperbolic product
structure} if there exist a continuous family of local unstable manifolds
$\Gamma^u=\{\gamma^u\}$ and a continuous family of local  stable manifolds
$\Gamma^s=\{\gamma^s\}$ such that
\begin{itemize}
    \item[i)] $\Lambda=(\cup \gamma^u)\cap(\cup\gamma^s)$;
    \item[ii)] $\dim \gamma^u+\dim \gamma^s=\dim M$;
    \item[iii)] each  $\gamma^s$ meets each $\gamma^u$  in exactly one point;
    \item[iv)] stable and unstable manifolds are transversal with angles bounded
    away from~0.
\end{itemize}
\noindent
If $\Lambda\subset M$ has a product structure,  we say that
$\Lambda_0\subset \Lambda$ is an {\em $s$-subset} if
$\Lambda_0$ also has a product structure and its defining
families $\Gamma_0^s$ and $\Gamma_0^u$ can be chosen with
$\Gamma_0^s\subset\Gamma^s$ and $\Gamma_0^u=\Gamma^u$; {\em
$u$-subsets} are defined analogously. For convenience we shall use the following notation: given $x\in\Lambda$, let
$\gamma^{*}(x)$ denote the element of $\Gamma^{*}$ containing $x$,
for $*=s,u$. Also, for each $n\ge 1$ let $(f^n)^u$ denote the restriction
of the map $f^n$ to $\gamma^u$-disks and let $\det D(f^n)^u$ be the
Jacobian of $D(f^n)^u$.

We say that \(  f  \) admits a  \emph{Gibbs-Markov-Young (GMY) structure}
if there exist  a set \(  \Lambda  \) with hyperbolic product structure and  constants  $C>0$ and
$0<\beta<1$,  depending on $f$ and $\Lambda$, satisfying the   following additional properties:


\begin{enumerate}
 \item[
    (P$_0$)] \emph{Detectable}: $\leb_\gamma(\Lambda)>0$ for each $\gamma\in\Gamma^u$.
        \item[
    (P$_1$)] \emph{Markov}: there are pairwise disjoint $s$-subsets $\Lambda_1,\Lambda_2,\dots\subset\Lambda$ such
    that
    \begin{enumerate}
 \item $\leb_{\gamma}\big((\Lambda\setminus\cup\Lambda_i)\cap\gamma\big)=0$ on each $\gamma\in\Gamma^u$;
 \item for each $i\in\NN$ there is $R_i\in\NN$ such that $f^{R_i}(\Lambda_i)$ is $u$-subset,
         and for all $x\in \Lambda_i$
         $$
         f^{R_i}(\gamma^s(x))\subset \gamma^s(f^{R_i}(x))\qand
         f^{R_i}(\gamma^u(x))\supset \gamma^u(f^{R_i}(x)).
         $$
    \end{enumerate}
\end{enumerate}


 \begin{enumerate}
\item[(P$_2$)] \emph{Contraction on stable leaves}: for all $\gamma^s\in\Gamma^s$, $x,y\in\gamma^s$ and $n\ge1 $
  $$\dist(f^n(y),f^n(x))\le C\beta^n .$$ 
 \end{enumerate}


 \begin{enumerate}
\item[(P$_3$)] \emph{Backward contraction on unstable leaves}:
for all $\gamma^u\in\Gamma^u$, $x, y \in \Lambda_i\cap\gamma^u$  and $0\leq n<R_i$
  $$
\dist(f^n(y),f^n(x))\le C\beta^{R_i-n}\dist(f^{R_i}(x), f^{R_i}(y)).
$$
 \end{enumerate}


 \begin{enumerate}
\item[(P$_4$)]  \emph{Bounded distortion}:  for all $\gamma^u\in\Gamma^u$ and $x, y \in \Lambda_i\cap\gamma^u$
  $$
\log\frac{\det D(f^{R_i})^u(x)}{\det D(f^{R_i})^u(y)}\leq C
\dist(f^{R_i}(x),f^{R_i}(y)).
$$
 \end{enumerate}


\begin{enumerate}
\item[(P$_5$)] \emph{Regularity of the foliations}:
\begin{enumerate}
 \item for all $\gamma^s\in\Gamma^s$, $x,y\in\gamma^s$ and $n\ge1 $
 $$
\log\prod_{i=n}^\infty\frac{\det Df^u(f^i(x))}{\det
Df^u(f^i(y))}\leq C \beta^n;
$$
 \item  given
$\gamma,\gamma'\in\Gamma^u$, we define
$\Theta\colon\gamma\cap\Lambda\to\gamma'\cap\Lambda$  by taking
$\Theta(x)$ equal to $\gamma^s(x)\cap \gamma'$. Then $\Theta$ is
absolutely continuous and
        $$\displaystyle \frac{d(\Theta_*\leb_{\gamma})}{d\leb_{\gamma'}}(x)=
        \prod_{i=0}^\infty\frac{\det Df^u(f^i(x))}{\det
        Df^u(f^i(\Theta^{-1}(x)))}.$$
\end{enumerate}
\end{enumerate}

 We
define a return time function
 $R: \Lambda\to\NN$ by $R|_{\Lambda_i}=
R_i$ and we say that the GMY structure has \emph{integrable return times} if  for
some (and hence all) $\gamma \in\Gamma^u$, we have
\begin{equation}\label{integrab_return}
    \int_{\gamma\cap\Lambda}R d\leb_\gamma <\infty.
\end{equation}

Most of the paper will be dedicated to the proof of the following result on the existence of GMY structures, from which we will deduce the other results, and whose interest goes beyond the applications presented here.

\cmt\label{thE}
Let \( f: M\to M \) be a  \( C^{1+} \)
diffeomorphism,  $K \subseteq M$   a forward invariant compact set on which
\(  f   \) is  partially hyperbolic with splitting \(  T_KM=E^s\oplus E^{cu}\), and $H \subseteq K$ a set with positive Lebesgue measure for which
$f$ is weakly non-uniformly expanding  along \(  E^{cu}\).
If there exists a  closed invariant transitive set \(  \Omega  \) such that
  \(  \omega(x)=\Omega \) for every \(  x\in H  \), then there exists a GMY structure \(  \Lambda\subseteq \Omega  \) with integrable return times.
  \fmt

 \subsection{Liftable measures}
 \label{s:lift}
Associated to a GMY structure we have an
 \emph{induced map} \(  F: \Lambda\to \Lambda  \)   defined by
\[
  F|_{\Lambda_{i}}= f^{R_{i}}|_{\Lambda_{i}}.
    \]
It is well known that  \(  F  \) has a unique  SRB measure \(  \nu   \).  Assuming  the integrability condition~\eqref{integrab_return} we can define the measure
\begin{equation}\label{lift.measure}
{\hat\mu}=
\sum_{j=0}^{\infty}f^j_{*}(\nu|\{  R >
j\}),
\end{equation}
which is a finite measure whose normalization \(\mu\)  is an
 SRB measure for \(  f  \), see \cite[Section 2]{You98}.

  Probability measures \( \mu \) obtained from GMY structures through  \eqref{lift.measure} often give a substantial amount of information on the statistical properties of the dynamics with respect to \(  \mu  \) such as decay of correlations, large deviations, limit theorems, etc.; see for instance \cite{You98, Gou05, MN05, MN08, RBY08}.
 This    motivates a general question as to whether any SRB measure \( \mu \) is of the form \eqref{lift.measure} for a suitable GMY structure. In this case we say that  \(  \mu  \) is \emph{liftable} (to a GMY structure).
As a consequence of the techniques used here we will obtain that for the partially hyperbolic systems considered in this paper, every SRB measure is liftable.

\cmt\label{thF}
Let \( f: M\to M \) be a  \( C^{1+} \)
diffeomorphism and  $K \subset M$   a compact  forward invariant set on which
\(  f   \) is  partially hyperbolic with splitting \(  T_{K}M=E^{s}\oplus E^{cu}  \).
 Then   \(  \mu  \) is an SRB measure with positive Lyapunov exponents in the \(E^{cu} \) direction supported on \(  K  \)   if and only  if \(  \mu  \) is liftable to a GMY structure on~\(  K  \).
\fmt

We emphasize that the assumption on the dynamics in center-unstable direction in the statement of  Theorem \ref{thF} is that the map has
\emph{positive Lyapunov exponents} as in \eqref{poslyap} rather than the stronger \emph{weak non-uniform expansion} as in \eqref{NUE2} which we need to assume instead for Theorem \ref{thA}, see also discussion following Corollary \ref{th:1}. In both cases we need to construct a Gibbs-Markov-Young structure with integrable return times which implies the existence of an SRB measure. The key difference however is that in Theorem \ref{thF} we already have an SRB measure by assumption and we just need to show that the two measures coincide. This allows us to prove, using a non-trivial argument, that some power of \(  f  \) is weakly non-uniformly expanding and thus apply the same construction as in the proof of Theorem \ref{thA}.

The paper is organized as follows.
In Section \ref{SectionErgodicComponents} we give an abstract criterion for verifying that
at most a finite number of topological attractors exist for a given set. In Section
\ref{sec:transatt} we show that this criterion is satisfied by the set \( H \) in Theorem \ref{thA}, thus proving item (1).
We then prove Theorem \ref{thE} in Sections \ref{s.partion}, \ref{s.measureHS}, \ref{GMY}
and \ref{intgrreturns}, which in particular applies to the sets \( \Omega\) obtained in item (1) of Theorem \ref{thA}.
 Combining the conclusion of Theorem \ref{thE} with the comments in the beginning of Section \ref{s:lift}
we get item (2) of Theorem \ref{thA} and the ``if'' direction of Theorem \ref{thF}. We prove the other implication of Theorem \ref{thF}  in Section \ref{lift}.

\section{Ergodic components}
\label{SectionErgodicComponents}

Let \(  X  \) be a compact metric space and \(  \mu  \) a Borel probability measure on \(  X  \).
Let \(  f: X \to X  \) be a measurable map, not necessarily preserving the measure $\mu$.
Given $x\in X$, the \emph{stable set} of $x$ is
\[
W^{s}(x)=\left\{y\in X:
\dist(f^j(x),f^j(y))\to0,\,\text{as $j\to\infty$}\right\}.
\]
Notice that the relation \(  x\sim  y  \) if and only if \(  y \in W^{s}(x)  \) defines an equivalence relation on  \(  X  \). In particular, we will use below the transitive property of this relation.
 If $U\subset X$, let
 \[
 W^{s}(U)=\bigcup_{x\in U}W^{s}(x).
 \]
 We recall that a set \(  U\subseteq X  \) is \emph{invariant} if \(  f^{-1}(U) = U  \).
 We now formulate a notion which is key to our argument.
We say that \(  Y \subseteq   X  \) is \( \mu\)-\emph{unshrinkable} if it is an  invariant set  with \(  \mu(Y) >0 \) and there exists a \(  \delta>0  \) such that for
every  invariant  set $U\subseteq Y$ we have
 \[
 \mu(U)>0 \quad \Rightarrow \quad \mu(W^{s}(U))>\delta.
 \]

\begin{Proposition} \label{PropositionFatErgodicAttractors}
Suppose \(  Y\subseteq X \) is \(  \mu  \)-unshrinkable. Then there
exists a finite number of  closed invariant subsets \(  \Omega_{1},...,\Omega_{\ell}   \) of \(  X  \) such that   for \(  \mu  \) almost every   \(  x\in Y  \) we have \(  \omega(x)=\Omega_{j}  \), for some \(  1\leq j \leq \ell  \).
\end{Proposition}

We will split the proof of Proposition \ref{PropositionFatErgodicAttractors} itself into two lemmas. To do this we need to introduce some additional concepts.
 We say that a set \(  S  \) is \emph{\(  s  \)-saturated} if $W^{s}(S)=S $.
 We say that $ S$ is a \emph{$u$-ergodic component} if it is  invariant,   \(  s  \)-saturated, and any subset \(  S'\subset S  \) which is also   invariant and  \(  s  \)-saturated, satisfies $\mu(S)=\mu(S')$ or $\mu(S')=0$.

\begin{Lemma}\label{LemmaCriteionForErgodicity}
Suppose  \(  Y\subseteq X \) is \(  \mu  \)-unshrinkable.
Then \(  Y  \) is contained (\( \mu\) mod 0) in the union  of
 a finite number of $u$-ergodic components.
\end{Lemma}

\dem
Let $Y_1 = Y$ and let
\[
\mathcal{F}(Y_1):=\{W^{s}(U)\,: \,U\subseteq Y_1, \ f^{-1}(U)=U,  \text{ and } \mu(W^{s}(U))> 0 \,\}
\]
Note that  $\mathcal{F}(Y_1)$ is
non-empty because $W^{s}(Y_1)\in\mathcal{F}(Y_1)$.  Moreover, we claim that
 \begin{equation}\label{eq:diff}
   W, W' \in  \mathcal{F}(Y_1) \qand
     \mu (W\setminus W') >0  \quad
     \Rightarrow
     \quad
       W\setminus W'\in \mathcal{F}(Y_1)
     \end{equation}
To see this, let \(  U, U'\subseteq Y_{1}  \) be invariant sets such that \(  W=W^{s}(U)\) and \( W'=W^{s}(U')  \). We claim that
\begin{equation}\label{eq:diff1}
  W\setminus W'
= W^{s}(U\setminus W^{s}(U')),
\end{equation}
Notice that    \(  U\setminus W^{s}(U') \subseteq Y_{1}  \), and also \(  U\setminus W^{s}(U')    \) is invariant
because both  \(  U, W^{s}(U')  \) are  invariant. Therefore
\eqref{eq:diff1} implies \eqref{eq:diff}.
To prove \eqref{eq:diff1}, we prove first of all that \(   W\setminus W'
\subseteq  W^{s}(U\setminus W^{s}(U')). \) Suppose \(  x\in W\setminus W'  \), i.e. \(  x\in W^{s}(U)  \) and \(  x\notin W^{s}(U')  \). This means that there exists \(  u\in U  \) such that \( x\in W^{s}(u) \) and also that  \(  x\notin W^{s}(u')  \) for any \(  u'\in U'  \), which implies that \(  x\notin W^{s}(z)  \) for any \(  z\in W^{s}(U')  \) by the transitivity of the equivalence relation \(  \sim  \) mentioned above. This proves the first inclusion. To prove
\(   W\setminus W'
\supseteq  W^{s}(U\setminus W^{s}(U')) \), let \(  x\in W^{s}(U\setminus W^{s}(U'))   \). Then clearly \(  x\in W  \) and
\(  x\in W^{s}(y)  \) for some \(  y\in U\setminus W^{s}(U')  \).
 It just remains to show that \(  x\notin W'  \).  Arguing by contradiction, suppose that \(  x\in W'=W^{s}(U')  \), then \(  x\in W^{s}(u')  \) for some \(  u'\in U'  \) and so, as \(  x \sim y  \), we have \(  y\in W^{s}(u')  \) which contradicts the fact that
\(   y\in U\setminus W^{s}(U')   \). This completes the proof of \eqref{eq:diff1} and hence of \eqref{eq:diff}.

Now consider the partial order on $\mathcal{F}(Y_1)$ defined by \emph{strict inclusion}, meaning that
\(  W \succ W'  \) if \(  W\supset W'  \)  and \(  \mu(W\setminus W')>0  \).
We claim that for this partial order relation,
every  totally ordered subset of $\mathcal{F}(Y_1)$ is finite,
and   in particular  it has a lower  bound.
Indeed, arguing by contradiction, suppose that there is an infinite sequence
\(  W_1\succ W_2\succ \cdots \) in $\mathcal{F}(Y_1)$, i.e.
$ W_1\supset W_2\supset\cdots$   with
$\mu(W_{k}\setminus W_{k+1})>0$, for all $ k\geq 1$.
Then  $$\sum_{k\ge1}\mu(W_{k}\setminus W_{k+1})=\mu(W_{1})<\infty,$$
and therefore  \(  \mu(W_{k}\setminus W_{k+1}) \to 0  \) as \(  k\to \infty  \).
Since $W_k\setminus W_{k+1}\in \cf(Y_1)$ by \eqref{eq:diff},  this contradicts our assumptions that \(  Y_{1}=Y  \) is \(  \mu  \)-unshrinkable. This shows that every totally ordered subset of $\mathcal{F}(Y_1)$ has a lower bound. Thus by Zorn's
Lemma
there exists at least one minimal element $W^{s}(U_1)\in\mathcal{F}(Y_1)$,
which therefore must necessarily be a $u$-ergodic component.

 We now let  $Y_2:=Y_{1}\setminus W^{s}(U_1)$, which is again invariant. If \(  \mu(Y_{2})=0  \) then \(  Y=Y _{1} \) is essentially contained in \(  W^{s}(U_{1})  \), which is a \( u  \)-ergodic component,
  and thus we are done. On the other hand, if \(  \mu(Y_{2})>0  \) we can repeat the entire argument above to obtain a
  set \(  U_{2}\subseteq Y_{2}  \) and a \(  u  \)-ergodic component \(  W^{s}(U_{2})  \).
 Inductively, we then  construct a collection of  disjoint
$u$-ergodic components $W^{s}(U_1),..., W^{s}(U_r)$ and continue as long as  $\mu(Y\setminus W^{s}(U_1)\cup
...\cup W^{s}(U_r))>0$. But, as $\mu(W^{s}(U_j))\geq \delta$ for all \(  1\leq j \leq r  \) by the assumption that \(  Y  \) is \(  \mu  \)-unshrinkable,  this process will stop and we will get the  conclusion.
\cqd

\begin{Lemma}\label{PropositionErgodicComponents2}
Suppose  $ S\subseteq X$ is a $u$-ergodic component. Then  there exists a closed invariant set
$\Omega\subseteq X$  such that ${\omega}(x)=\Omega$ for \(  \mu  \)-almost every
$x\in S$.
\end{Lemma}

\dem
 Given any open set $B\subset X$,
let
\[
B_{\omega} :=\{x\in S: \omega(x)\cap B\ne\emptyset\}.
\]
Then \(   B_{\omega}  \) is invariant and \(  s  \)-saturated and therefore, by the assumption that \(  S  \) is \(  u  \)-ergodic,   \(  \mu (B_{\omega})=0  \) or \( \mu(B_{\omega})= \mu(S)  \).
Now, let $Z_1=X$ and $\mathcal{C}_1$ be any finite  covering  of $X$ by open balls of radius  $1$. By the previous considerations, for every \(  B\in\mathcal {C}_{1}  \) we have \( \mu (B_{\omega})=0  \) or \( \mu (B_{\omega})=\mu(S)  \) and therefore, since we only have a finite number of elements in \(  \mathcal {C}_{1}  \),
   there exists  at least one \(  B_{\omega}\in \mathcal {C}_{1}  \) such that \(  \mu  (B_{\omega})= \mu(S)  \).
 Let
 \[
 \mathcal C_{1}'=\{B\in\mathcal C_1: \mu (B_{\omega})=0\} \qand
 Z_2=Z_1\setminus\bigcup_{B\in\mathcal C_1'}B.
 \]
 Then  \(  Z_{2}  \) is a non-empty compact  set and \(  \omega(x)\subseteq Z_{2}  \) for \(  \mu  \)-almost every \(  x\in S  \).  We can therefore repeat the procedure with a finite cover
 $\mathcal C_{2}$ of $Z_{2}$ by  open balls of radius  $1/2$, and,
by induction, construct sequences $ \mathcal C_1,\mathcal C_2,\dots$, $ \mathcal C_1',\mathcal C_2',\dots$  and $Z_1, Z_2,\dots$
such that $Z_1\supset Z_2\supset ...$ is a
sequence of non-empty compact sets and
$\omega(x)\subset Z_{j}$ for  almost every $x\in S$.
In particular we have
\[
\omega(x)\subseteq {\Omega} :=\bigcap_{n\geq 1} Z_n
\]
It just remains to show that \(  \Omega  \subseteq\omega(x) \) for \(  \mu  \) almost every \(  x\in S  \).
Indeed, given  \(  y\in \Omega  \) we have that   \(  y \in  Z_{n}  \)
for every \(  n\geq 1  \),  and therefore there is some \(  B^{(n)}\in \mathcal C_{n}\setminus \mathcal C_{n}'  \)
such that \( y\in  B^{(n)}   \).
Since \(  \diam(B^{(n)})\to 0  \)  as $n\to\infty$, this implies that   $\bigcap_n B^{(n)}=\{y\}$.
Moreover, as \(  B^{(n)}\in \mathcal C_{n}\setminus \mathcal C_{n}'  \) we have that
 \(  \mu(B^{(n)}_{\omega})= \mu(S)  \) and therefore
 \(  \omega(x)\cap B^{(n)}\neq \emptyset  \)    for \(  \mu  \) almost all \(  x\in S  \).  This implies that
  \(  y\in\omega(x)  \)
    for \(  \mu  \) almost all \(  x\in S  \) and, as \(  \omega(x)  \) is closed and invariant, the statement follows.
\cqd

\section{Transitive attractors}\label{sec:transatt}

Here we prove the topological part of Theorem \ref{thA}.
We assume  throughout this section the assumptions of that theorem. Let
 \( f: M\to M \) be a  \( C^{1+} \)
diffeomorphism,  $K \subset M$   a forward invariant compact set on which
\(  f   \) is  partially hyperbolic, and $H \subseteq K$ a set with \(  \leb (H)>0  \) on which
$f$ is weakly non-uniformly expanding  along \(  E^{cu}\). The main result of this section is the following proposition.

\begin{Proposition}\label{thmAtop}
There exist  closed invariant  sets  \(  \Omega_{1},...,\Omega_{\ell}  \subseteq K\)
such that
  for Lebesgue almost every \(  x\in H  \) we have
  \(  \omega(x)=\Omega_{j}  \) for some  \(  1\leq j\leq \ell \). Moreover, each \(  \Omega_{j}  \) is transitive and contains a \(  cu  \)-disk \(  \Delta_{j}  \) of radius \(  \delta_{1}/4  \) on which \(  f  \) is weakly non-uniformly expanding along \(  E^{cu}  \) for \(  \leb_{\Delta_{j}}  \) almost every point in \(  \Delta_{j}  \).
\end{Proposition}

We first prove some preliminary lemmas.
We remark that \(  K  \) is not assumed to contain any open sets. We therefore
 fix continuous extensions of the
two sub-bundles $E^{s}$ and $E^{cu}$ to some compact neighborhood $ V$
of $K$, that we still denote $E^{s}$ and $E^{cu}$.
We do not require these extensions to be \(  Df  \) invariant.
Given $0<a<1$, we define the {\em centre-unstable cone field
$C_a^{cu}=\left(C_a^{cu}(x)\right)_{x\in V}$ of width $a$\/} by
\begin{equation}
\label{e.cucone} C_a^{cu}(x)=\big\{v_1+v_2 \in E_x^{s}\oplus
E_x^{cu}: \|v_1\| \le a \|v_2\|\big\}.
\end{equation}
We define the {\em stable cone field
$C_a^{s}=\left(C_a^{s}(x)\right)_{x\in V}$ of width $a$\/} in a
similar way, just reversing the roles of the sub-bundles in
(\ref{e.cucone}). We fix $a>0$ and $V$ small enough so that
 the domination condition
\eqref{domination1}  remains valid   in the
two cone fields:
$$
\|Df(x)v^{s}\|\cdot\|Df^{-1}(f(x))v^{cu}\|
<\lambda\|v^{s}\|\cdot\|v^{cu}\|
$$
for every $v^{s}\in C_a^{s}(x)$, $v^{cu}\in C_a^{cu}(f(x))$ and
any point $x\in V\cap f^{-1}(V)$. Note that the centre-unstable cone
field is forward invariant:
 $Df(x) C_a^{cu}(x)\subset
C_a^{cu}(f(x)),$ whenever $x,f(x)\in V$. Actually, the
domination property together with the invariance of ${E}^{cu}
|_{K}$ imply that
$
Df(x) C_a^{cu}(x) \subset C_{\lambda a}^{cu}(f(x))
                  \subset C_a^{cu}(f(x)),
$
for every $x\in K$, and this extends to any $x\in V\cap f^{-1}(V)$
just by continuity.

 Given $0<\sigma<1$, we say that $n$ is a
{\em $\sigma$-hyperbolic time\/} for  $x\in K$ if
$$
\prod_{j=n-k+1}^{n}\|Df^{-1} \mid E^{cu}_{f^{j}(x)}\| \le \sigma^k,
\qquad\text{for all $1\le k \le n$.}
$$
The next result gives the existence of (infinitely many) $\sigma$-hyperbolic times
for points satisfying the weak nonuniform expansion condition~\eqref{NUE2}.
For a proof   see \cite[Corollary~5.3]{AP}.
\cle\label{l:hyperbolic2}
   There are
     $\sigma > 0$ and $\theta > 0$  such that  if  \eqref{NUE2} holds  for  $x \in K$, then
\begin{equation*}
\limsup _{n\to\infty}\frac{1}{n}\#\{1\leq j\leq n: j \text{ is a }
\sigma\text{-hyperbolic time for }  x\} \geq
 \theta.
 \end{equation*}
\fle

Under the stronger assumption \eqref{NUE1}, it was proved in \cite{ABV} that there is positive frequency at infinity of hyperbolic times, which means taking ``$\liminf$'' instead of ``$\limsup$'' in Lemma \ref{l:hyperbolic2}. The positive frequency of $\sigma$-hyperbolic times at infinity plays a crucial role in the argument used in \cite[Corollary 3.2]{ABV} to prove the existence of SRB measures, and this is the reason why we cannot use those arguments here.

Hyperbolic times are defined pointwise but, as we shall see below, some important properties can be derived  for a neighbourhood of the reference point at a hyperbolic time. From now on we fix \(\sigma\) and \(\theta\) as in Lemma
\ref{l:hyperbolic2}.
Now observe that,
by continuity of the derivative,
we can choose   $a>0$ and $\delta_1>0$ sufficiently  small so that the
$\delta_1$-neighborhood of $K$ is contained in $V$ and
\begin{equation}\label{e.delta1}
 \|Df^{-1}(f(y)) v \| \le {\sigma^{-1/4}}
\|Df^{-1}|E^{cu}_{f(x)}\|\,\|v\|
\end{equation}
for all  $x\in K, y\in V$ with $\dist(x,y)\le \delta_1$ and $v\in
C^{cu}_a(y)$.  From now on we fix these values of \( a, \delta_1\) so that \eqref{e.delta1} holds.

We say that an embedded $C^1$ submanifold $D\subset V$ is a {\em
\( cu \)-disk}  if the tangent
subspace to~$D$ at each point $x\in D$ is contained in the
corresponding cone $C_a^{cu}(x)$. Then $f(D)$ is also a \(  cu  \)-disk, if it is contained in $V$, by the
domination property.
Given any disk $D\subset M$, we use $\dist_D(x,y)$ to
denote the distance between $x,y\in D$, measured along
$D$.

\cle \label{l.contraction2} Let $D$ be a $cu$-disk. There exists  $C_{1}>1$  such that if $n$ is a $\sigma$-hyperbolic time for $x\in K\cap D$, then
 there exists a neighbourhood \(  V_{n}^{+}(x)  \) of \(  x  \) in $D$ such that
 $f^{n}$ maps $V_n^{+}(x)$ diffeomorphically onto a
\(  cu  \)-disk  \(  B^u_{2\delta_1}(f^n(x))   \) of radius $2\delta_1$ around  $f^{n}(x)$. Moreover,  for every $1\le k
\le n$ and $y, z\in V_n(x)^+$ we have

\begin{enumerate}
    \item
$\displaystyle
\dist_{f^{n-k}(V^+_n(x))}(f^{n-k}(y),f^{n-k}(z)) \le
\sigma^{3k/4}\dist_{f^n(V^+_n(x))}(f^{n}(y),f^{n}(z));
$
\item
$\displaystyle
\log \frac{|\det Df^{n} \mid T_y D|}
                     {|\det Df^{n} \mid T_z D|}
            \le C_{1} \dist_{f^n(D)}(f^{n}(y),f^{n}(z));
$
\item for any
Borel sets $X, Y\subset V^+_n(x)$,
$$
\frac{\leb_{f^{n}(V^+_{n}(x))}(f^n(X))}{\leb_{f^{n}(V^+_{n}(x))}(f^n(Y))} \leq C_1\frac{\leb_{V^+_{n}(x)}(X)}{\leb_{V^+_{n}(x)}(Y)}.
$$
\end{enumerate}
\fle
The first two items are  proved in
\cite[Lemma 2.7 \&  Proposition 2.8]{ABV}, and the third item is a standard consequence of the second one.
Notice   that  the factor  $\sigma^{3/4}$ in the first item of this lemma  differs from the factor $\sigma^{1/2}$ in \cite[Lemma 2.7]{ABV} simply because we have chosen $\delta_1>0$ sufficiently small so that \eqref{e.delta1} holds, contrarily to estimate (6) in \cite{AP} where  $\delta_1>0$ is chosen so that a similar conclusion holds with $\sigma^{1/2}$ in the place  of  $\sigma^{1/4}$.

\cre\label{re.bounded}
Notice that if we replace the assumption that \( n \) is a \(\sigma\)-hyperbolic time in   Lemma~\ref{l.contraction2}  with the assumption that \( n \) is a \(\sigma^\alpha\)-hyperbolic time for some \(\alpha>1/4\), then the conclusions of the lemma continue to hold with \(\sigma^{\alpha-1/4}\) instead of \(\sigma^{3/4}\) in item (1), where the term \(1/4 \) comes from \eqref{e.delta1}.
\fre

Now we   define
\(
  V_{n}(x)\subseteq V^{+}_{n}(x),
  \)
where $f^{n}(V_{n}(x))=  B^u_{\delta_1}(f^n(x))
$
is the \(  cu  \)-disk of radius \(  \delta_{1}  \) around \(  f^{n}(x)  \) contained in \( B^u_{2\delta_1}(f^n(x))\) as in Lemma \ref{l.contraction2}.
The
sets \( V_n(x) \)  are called \emph{hyperbolic pre-disks} and their
images \( f^{n}(V_n(x)) \)  \emph{hyperbolic disks}.   The following result is proved in \cite[Proposition 5.5]{AP}.

\cle \label{posdens}
Let \(  D \) be a \(  cu  \)-disk and \(  U\subseteq H  \) with \(  \leb_{D}(U)>0  \).
Then there exists a sequence of sets \(  ...\subseteq W_{2}\subseteq W_{1}\subseteq D  \)
and a sequence of integers \(  n_{1} < n_{2}<\cdots   \) such that:
\begin{enumerate}
\item \(  W_{k}  \) is contained in some hyperbolic pre-disk with hyperbolic time \(  n_{k}  \);
\item \(  D_{k}:=f^{n_{k}}(W_{k})  \) is a \(  cu  \)-disk of radius \(  \delta_{1}/4  \);
\item
\(\displaystyle
\lim_{k\to\infty} \frac{\leb_{D_{k}}f^{n_{k}}(U\cap D)}{\leb_{D_{k}}(D_{k})}=1
\).
\end{enumerate}

\fle

Now we are in conditions to prove Proposition~\ref{thmAtop}. We define
\[
\widetilde H:= \bigcup_{n\in \mathbb Z} f^{n}(H).
\]
Then \(  \widetilde H  \) is clearly invariant and \(  \leb(\widetilde H)>0  \).

\cle \label{le.unshrink}
\( \widetilde H  \) is \(  \leb  \)-unshrinkable.
\fle
\dem
To prove that \(  \widetilde H  \) is \(  \leb  \)-unshrinkable it is   sufficient to show that there exists \(  \delta>0  \) such that for every \(  f  \)-invariant set
 \(  U\subseteq \widetilde H  \) with  \(  \leb(U)>0  \) we have \(  \leb(W^{s}(U))> \delta  \).
 We remark that in the proof of this assertion,
  to be given in the following paragraphs, we will only use the assumption that \(  U  \)
 is forward invariant. This allows us to  assume without loss of generality that \(  U\subseteq K  \).
  Indeed,  if  \(  U  \) is invariant with positive Lebesgue measure, then it must intersect \(  K  \) in a set of positive Lebesgue measure and, as \(  K  \) is forward invariant, also \( U \cap K  \) is forward invariant. Clearly, if \(  \leb(W^{s}(U\cap K))>\delta  \) then we also have \(  \leb(W^{s}(U))>\delta  \). In particular, as \(  U\subseteq K  \) it admits a partially hyperbolic structure and, as also \(  U\subseteq \widetilde H  \), it  is weakly non-uniformly expanding  along \(  E^{cu}\).

Now we show that there exists a \(  cu  \)-disk \(  D\subseteq V  \) such that \(  \leb_{D}(U)>0  \).
Recall that \(  V  \) is the neighbourhood of \(  K  \) introduced in the beginning of this section. To see this, consider a Lebesgue density point \(  p  \) of \(  U  \).
Notice that  \(  T_{p}M  \) has a partially hyperbolic splitting \(  E^{s}_{p}\oplus E^{cu}_{p}  \) and we can consider a neighbourhood of the origin foliated by disks parallel to the \(  E^{cu}  \) subspace whose images under the exponential map \(  \exp_{p}  \) are \(  cu  \)-disks in the manifold.
  Since \(  \exp_{p}  \) is a local diffeomorphism, the preimage of \(  U  \) under the exponential map has positive volume in \(  T_{p}M  \) and full density in the origin. By Fubini at least one of the disks above must intersect this set in positive relative volume, and the same must hold for its image under the exponential map.

Now let  \(  D\subseteq V  \)  be a \(  cu  \)-disk satisfying \(  \leb_{D}(U)>0  \), as in the previous paragraph.
Consider the sequences \(  \cdots \subseteq W_{2}\subseteq W_{1}\subseteq D  \) and \(  n_{1}<n_{2}<\cdots   \) given by Lemma \ref{posdens}.
By the third item of Lemma \ref{posdens} it follows that the relative measure of  \(  f^{n_{k}}(U\cap D)  \) in
\(  D_{k}  \) converges to~1. Since \(  U  \) is forward invariant we conclude that the relative measure of \(  U  \) in \(  D_{k}  \) converges to~1 and therefore \(  \leb_{D_{k}}(U)\to \delta_{1}/4  \) as \(  k \to \infty  \).
Since  \(  U \subseteq K  \),  all points of \(  U  \) have local stable manifolds of uniform size and the foliation defined by these local stable manifolds is absolutely continuous, it follows
  that \(  \widetilde H  \) is \(  \leb  \)-unshrinkable.
  \cqd

 The previous result, together with Proposition \ref{PropositionFatErgodicAttractors}, imply that there exist closed invariant sets \(  \Omega_{1},...,\Omega_{\ell}  \) such that for Lebesgue almost every \(  x\in H  \) we have
  \(  \omega(x)=\Omega_{j}  \) for some  \(  1\leq j\leq \ell \).  This gives  the first assertion of Proposition \ref{thmAtop}. We leave the proof of the remaining part of Proposition
\ref{thmAtop} to the next two lemmas.

\cle Each  \(  \Omega=\Omega_{j}  \) contains a \(  cu  \)-disk \(  \Delta  \) of radius \(  \delta_{1}/4  \) on which \(  f  \) is weakly non-uniformly expanding along \(  E^{cu}  \) for \(  \leb_{\Delta}  \) almost every point in \(  \Delta  \).
\fle 
\dem
 Let
  \[
A^{(n)}=\{x\in H: \dist(f^{k}(x), \Omega) \leq 1/n \text{ for every } k \geq 0\}.
\]
Since the set of points \(  x\in H  \) with \(  \omega(x)=\Omega  \) has positive Lebesgue measure, we clearly have \(  \leb(A^{(n)})>0  \) for every \(  n\geq 1  \).   Then, by the same arguments used in the proof of Lemma~\ref{le.unshrink}, with \(  A^{(n)}  \) playing the role of \(  U  \), there exists  a \(  cu  \)-disk
 \(  D^{(n)}\subseteq V  \) such that \(  \leb_{D^{(n)}}(A^{(n)})>0  \), and corresponding
 sequences \(  \cdots \subseteq W_{2}^{(n)}\subseteq W_{1}^{(n)}\subseteq D^{(n)}  \), \(  n_{1}<n_{2}<\cdots   \) (also depending on \(  n  \), but we omit the superscript here for obvious reasons...) and
 \(  cu  \)-disks \( D^{(n)}_{k}=f^{n_{k}}(W^{(n)}_{k})  \)
 such that
 \begin{equation}\label{istoaqui}
   \leb_{D^{(n)}_{k}}(A^{(n)})\to \delta_{1}/4, \quad   \text{ as } k\to\infty.
   \end{equation}
Let  \(  p^{(n)}_{k}  \) denote the center of each disk $D_{k}^{(n)}$.
 Up to taking a subsequence, we may assume that the sequence  \(  \{p^{(n)}_{k}\}  \) converges to a point  \(  p^{(n)}\in K  \), and
up to taking a further subsequence, and using Ascoli-Arzel\`a and the fact that the disks \(  D^{(n)}_{k}  \) have  tangent directions contained in the \(  cu  \)-cones, we may assume  that the sequence \(  \{D^{(n)}_{k}\}  \) converges uniformly, as \(  k\to\infty  \), to some  \(  cu  \)-disk~\(  \Delta^{(n)}  \) of radius \(  \delta_{1}/4  \).  Notice that each \(  \Delta^{(n)}  \) is  necessarily contained in a neighbourhood of \(  \Omega  \) of radius \(  1/n  \).

We claim the \(  f  \) is weakly non-uniformly expanding along \(  E^{cu}  \)
for  \(  \leb_{\Delta^{(n)}}  \)
almost every  point in \(  \Delta^{(n)}  \).
To see this, recall first of all that the property of weak non-uniform expansion is an asymptotic property and therefore if it is satisfied by a point \(  x  \), then it is satisfied by every point \(  y\in W^{s}(x)  \).
Moreover, every point of \(  \Delta^{(n)}  \) has a local stable manifold of uniform size, and   the foliation by  those local stable manifolds is absolutely continuous.   Since the sequence
 \(  \{D^{(n)}_{k}\}  \) converges uniformly to \(  \Delta^{(n)}  \), for large \(  k  \), the disks \(  D^{(n)}_{k}  \) will intersect the stable foliation through  points of \( \Delta^{(n)}  \), and therefore,
 by \eqref{istoaqui} and the fact that \(  A^{(n)}\subseteq H  \), it follows that  \(  f  \) is weakly non-uniformly expanding along \(  E^{cu}  \)
for  \(  \leb_{\Delta^{(n)}}  \)
 almost every point in \(  \Delta^{(n)}  \).

Now, arguing as above, we can consider a subsequence of  \(  \Delta^{(n)}  \)'s converging uniformly to some \(  cu  \)-disk
\(  \Delta  \) of radius \(  \delta_{1}/4  \) and \(  f  \) is weakly non-uniformly expanding along \(  E^{cu}  \)
for  \(  \leb_{\Delta}  \)
 almost every point in \(  \Delta  \). As each \(  \Delta^{(n)}  \) is   contained in a neighbourhood of \(  \Omega  \) of radius \(  1/n  \) and \(  \Omega  \) is closed, it follows that \(  \Delta\subseteq \Omega  \).
 \cqd

\cle
 \( f|_{\Omega}  \) is transitive.
  \fle
  \dem
  Recall that by construction there exists some point (in fact a positive Lebesgue measure set of points) in \(  H  \) whose \(  \omega  \)-limit set coincides with \(  \Omega  \). The orbit of any such point must eventually hit the stable manifold of some point in \(  \Delta\subseteq \Omega  \). As points in the same stable manifold have the same \(  \omega  \)-limit sets, we conclude that there exists a point of \(  \Omega  \) whose orbit is dense in \(  \Omega  \).
\cqd

\section{Construction on a reference leaf}\label{s.partion}

In this section we describe an algorithm for the construction of a partition of some subdisk of \(  \Delta  \) which is the basis of the construction of the GMY structure.
We first fix some arbitrary \(  1\leq j \leq \ell  \) and for the rest of the paper we let \(  \Omega=\Omega_{j}  \) and \(  \Delta=\Delta_{j}  \) as in Proposition~\ref{thmAtop}. We also fix a constant \(\delta_s>0\) so that local stable manifolds \( W^s_{\delta_s}(x)\) are defined for all points \(x\in K \).
For any subdisk  \(  \Delta'\subset \Delta  \)  we define
$$
 \cc(\Delta') =\bigcup_{x\in\Delta'
}W^s_{\delta_s}(x).
$$
Let $\pi$ denote the projection  from $\cc(\Delta')$ onto $\Delta'$
along local stable leaves.
We say that a centre-unstable disk $\gamma^u\subset M$ {\em $u$-crosses}
$\cc(\Delta')$ if $\pi( \gamma )=\Delta'$ for some connected component  \(  \gamma  \) of \(  \gamma^{u}\cap \cc(\Delta')  \).

\cre\label{pretend}
We will often be considering \(  cu  \)-disks which \(  u  \)-cross \(  \mathcal C(\Delta')  \).
By continuity of the stable foliation, choosing \(  \delta_{s}  \) sufficiently small, the diameter and  Lebesgue measure of such disks intersected with $\mathcal C(\Delta')$ are very close to those of \(  \Delta'\), respectively. To simplify the notation and the calculations below we will ignore this difference as it has no significant effect on the estimates.
\fre

\cle\label{l.delta2} Given $N\in\NN$, there exists $\delta_2=\delta_2(N,\delta_1)>0$ such that if $\gamma^u\subset \Omega$
is a $cu$-disk of radius~$\delta_1/2$ centred at $z$, then
$f^m(\gamma^u)$ contains a $cu$-disk of radius
$\delta_2$ centred at $f^m(z)$, for each $1\le m\le N$. \fle

\dem  We first prove the result for $j=1$. Let $z$ be the center of $\gamma^u$. Let  $f(y)$ be a point in
$\partial f(\gamma^u)$ minimizing the distance from $f(z)$ to
$\partial f(\gamma^u)$, and let $\eta_1$ be a curve of minimal
length in $f(\gamma^u)$ connecting $f(z)$ to $f(y)$. Letting
$\eta_0=f^{-1}(\eta_1)$ and $\dot\eta_1(x)$ be the tangent vector to the curve $\eta_1$
at the point $x$, we have
$$
\|Df^{-1}(w) \dot\eta_1(x)\| \le C \, \|\dot\eta_1(x)\|,
$$
where
 $$C=\max_{x\in M}\big\{\|Df^{-1}(x)\|\big\}\ge1.$$
  Hence,
$$
\length(\eta_0) \le C\length(\eta_1) .
$$
Noting that $\eta_0$ is a curve connecting $z$ to $y\in \partial
\gamma^u$, this implies that $\length(\eta_0)\ge \delta_1/2$, and so
$$
\length(\eta_1) \ge C^{-1}\length(\eta_0)\ge C^{-1}\delta_1/2.
$$
Thus  $f(\gamma^u)$ contains the  $cu$-disk $\gamma^u_1$ of radius
$C^{-1}\delta_1/2$ around $f(z)$.

 Making now $\gamma^u_1$ play the role of $\gamma^u$ and $f^2(z)$ play the
 role of $f(z)$, with the argument above we prove  that  $f(\gamma^u_1)$
 contains a $cu$-disk of radius $C^{-2}\delta_1/2^2$ centered at $f^2(z)$.
 Inductively, we  prove that
 $f^m(\gamma^u)$ contains a $cu$-disk of radius $C^{-m}\delta_1/2^m\ge C^{-N}\delta_1/2^{N}$
 around $f^m(z)$, for each $1\le m\le N$. We  take
 $\delta_2=C^{-N}\delta_1/2^{N}$.
 \cqd


\cle\label{l.N0q}
There are   $p\in \Delta$ and $N_0\ge 1$ such that for all $\delta_0>0$ sufficiently small
and each  hyperbolic pre-disk $V_n(x)\subseteq \Delta$ there is $0\le m\le N_0$ such that  $f^{n+m}(V_n(x))$ intersects $W^s_{\delta_s/2}(p)$ and $u$-crosses $\mathcal C(B^u_{\delta_0}(p))$, where $B^u_{\delta_{0}}(p)$  is the
ball in $\Delta$ of radius $\delta_0$  centred at \( p \).
\fle
\dem First of all we observe that, as the sub-bundles in the dominated spliting have angles uniformly bounded away from zero,  given any $\rho>0$ there is $\alpha=\alpha(\rho) >0$,
with $\alpha\to0$ as $\rho\to0$,
for which the following holds:
if $x,y\in  \Omega$  satisfy
$\dist(x,y)<\rho$ and $\dist_{\gamma^u}(y,\partial \gamma^u)>\delta_1$ for some $cu$-disk $\gamma^u\subset\Omega$, then $W^s_{\delta_s}(x)$ intersects
$\gamma^u$ in a point $z$ with
$$\dist_{W^s_{\delta_s}(x)}(z,x)<\alpha  \qand\dist_{\gamma^u}(z,y)<\delta_1/2.$$
Take $\rho>0$ small enough so that $4\alpha<\delta_s $. Since $f|_\Omega$  is transitive, we may choose
$q\in  \Omega$  and
$N_0\in\NN$ such that both:
\begin{enumerate}
\item $W^s_{\delta_s/4}(q)$ intersects
$\Delta$ in a point $p$ with $\dist_\Delta(p ,\partial \Delta)>0$; and
\item $\{ f^{-N_0}(q),\dots, f^{-1}(q),q\}$ is $\rho$-dense in
$ \Omega$. 
\end{enumerate}
Given a hyperbolic pre-disk $V_n(x)\subseteq \Delta$ we have by definition that  $f^{n}(V_n(x))$ is a $cu$-disk of radius $\delta_1$ centred at $y=f^n(x)$ inside $\Omega$.
Consider $0\le m\le
N_0$ such that $\dist(f^{-m}(q),y)<\rho$.  Then, by the choice of
$\rho$ and $\alpha $, we have that  $W^s_{\delta_s}(f^{-j}(q))$
intersects $f^{n}(V_n(x))$  in a point $z$ with
$\dist_{W^s_{\delta_s}(f^{-j}(q))}(z,f^{-j}(q))<\alpha  <\delta_s/4$
and $\dist_{f^{n}(V_n(x))}(z,y)<\delta_1/2$. In particular,
$f^{n}(V_n(x))$ contains a $cu$-disk $\gamma^u$ of radius $\delta_1/2$ centred at $z$. It follows from Lemma~\ref{l.delta2} that $f^m(\gamma^u)$  contains a $cu$-disk of radius
 $\delta_2=\delta_2(N_0, \delta_1)>0$ centered at $f^m(z)\in W^s(p)$. Moreover, as distances are not expanded under iterations of points in the same stable manifold, we have
  $$\dist_{W^s(p)}(f^m(z),p)\le \dist_{W^s(p)}(f^m(z),q)+\dist_{W^s(p)}(q,p)\le \frac{\delta_s}{4}+\frac{\delta_s}{4},
  $$
which means that  $f^{n+m}(V_n(x))$ intersects $W^s_{\delta_s/2}(p)$. Also, choosing $\delta_0>0$ sufficiently small (depending only on $\delta_2$) we have $u$-crosses $\mathcal C(B^u_{\delta_0}(p))$.
\cqd

We now fix \(  p\in \Delta, N_{0}\geq 1  \) and \( \delta_{0}>0  \) sufficiently small so that the conclusions of Lemma~\ref{l.N0q} hold.
Considering
the constant
\begin{equation}\label{eq:K0}
K_0 = \max_{x\in M}\left\{\|Df^{-1}(x)\|, \|Df(x)\|\right\},
\end{equation}
we choose in particular $\delta_0>0$  small so that
  \begin{equation}\label{eq.delta_0}
2\delta_0K_0^{N_0}\sigma^{-N_0} < \delta_1 K_0^{-N_0}.
\end{equation}
Now we define
\begin{equation}\label{Delta0}
\Delta_0=B^u_{\delta_{0}}(p)\qand  \mathcal C_0=\mathcal C(\Delta_{0}).
 \end{equation}
 We also choose $\delta_0>0$ small so that any $cu$-disk intersecting $W^s_{3\delta_s/4}$ cannot reach the top or bottom parts of $\mathcal C_0$, i.e. the boundary points of the  local stable manifolds $W^s_{\delta_s}(x)$ through points $x\in \Delta_0$.
For every  $n\ge 1$ we define
 $$
 H_n=\{x\in \Delta \cap H \colon \text{ $n$ is a hyperbolic time for
 $x$ }\}.
 $$
It follows from Lemma~\ref{l.contraction2} that for each   $x \in H_{n}\cap \Delta_0$ there exists a hyperbolic pre-disk $V_{n}(x)\subset \Delta$.
Then,  by Lemma \ref{l.N0q}  there are $0\le m\le N_0$
and a centre-unstable disk \(  \omega_{n}^{x}\subseteq \Delta  \) such that
\begin{equation}\label{D.candidate2}
  \pi(f^{n+m}(\omega_{n}^{x})) =\Delta_{0}.
\end{equation}
We remark that condition \eqref{D.candidate2} may in principle hold for several values of \(  m  \). For definiteness, we shall always assume that \(  m  \) takes the smallest possible value.
Notice that \(  \omega^{x}_{n}  \) is associated to \(  x  \) by construction, but does not necessarily contain \(  x  \).

%

In the sequel we describe  an inductive  partitioning algorithm which gives rise to a (\(  \leb   \) mod~0) partition \( \mathcal P \) of the \(  cu  \)-disk \(  \Delta_{0}  \).

\medskip

\paragraph{\bf First step of induction}
Notice  that since \(  \|Df\|  \) is uniformly bounded, for any \(  n\geq 1  \), all hyperbolic pre-disks \(  V_{n}(x)  \) contain a ball of some radius \(  \tau_{n}>0  \)  depending only on~\(  n  \). In particular, by compactness, the set \(  H_{n}\cap \Delta_0  \) is covered by a finite number of hyperbolic pre-disks \(  V_{n}(x)  \).
We fix  some large \(n_0\in\NN \)  and ignore any dynamics occurring
up to time \( n_{0} \).
Then  there exist \(  \ell_{n_{0}}  \) and points  $z_1,\dots,
z_{\ell_{n_0}}\in H_{n_0}$ such that
 \begin{equation*}
\label{eq:set} 
 H_{n_0}\cap \Delta_0\subset V_{n_0}(z_1)\cup\cdots \cup V_{n_0}(z_{\ell_{n_0}}).
 \end{equation*}
We now choose a maximal subset of points
$x_1,\dots,x_{j_{n_0}}\in\{z_1,\dots, z_{\ell_{n_0}}\}$
such that the corresponding sets \(  \omega^{x_{i}}_{n_{0}}  \) of type \eqref{D.candidate2} are pairwise disjoint and contained in \(  \Delta_{0}  \), and let
$$
\mathcal P_{n_0}=\{\omega_{n_0}^{x_1},\ldots,
\omega_{n_0}^{x_{j_{n_0}}}\}.
$$
These are the elements of the partition $\cp$ constructed in the
$n_0$-step  of the algorithm.
Let 
\begin{equation*}
\Delta_{n_0} = \Delta\setminus \bigcup_{\omega \in
\mathcal P_{n_0}}\omega.
\end{equation*}
 For each
\(0\leq i \leq j_{n_0}\), we define the inducing time
 \[
R|_{\omega^{x_i}_{n_0}}=n_0+m_i
\]
where \(  0\leq m_{i} \leq N  \) is the integer associated to \(  \omega^{x_{i}}_{n_{0}}  \) as in \eqref{D.candidate2}.
Let now $ Z_{n_0}$ be the set of points in $\{z_1,\dots,
z_{\ell_{n_0}}\}$ which were not chosen in the construction of
$\mathcal{P}_{n_0}$, i.e.
$$ Z_{n_0}=\{z_1,\dots, z_{\ell_{n_0}}\}\setminus\{x_1,\dots, x_{j_{n_0}}\}.$$
We remark that for every \(  z\in  Z_{n_{0}}  \), the set  \(  \omega_{n_{0}}^{z}  \) associated to \(  z  \) must either intersect some \(  \omega^{x_{i}}_{n_{0}}\in \mathcal P_{n_{0}}  \) or intersect the complement of \(  \Delta_{0}  \) in \(  \Delta  \), since otherwise it would have been included in the set \(  \mathcal P_{n_{0}}  \).
We now introduce some notation to keep track  of which one of the above reasons  is responsible for the fact that  \(  z \) belongs to \( Z_{n_0} \).
We let  $\Delta_{0}^c=\Delta\setminus\Delta_{0}$ and for each  $\omega\in \mathcal P_{n_0}\cup\{\Delta_{0}^{c}\}$ we define
\begin{equation*}\label{auxiliosatelite0}
 Z^{\omega}_{n_0}=\left\{x \in   Z_{n_0}: \omega^x_{n_0}\cap \omega \neq\emptyset\right\}
\end{equation*}
and  the  associated \emph{$n_0$-satellite} set
\begin{equation*}
    S_{n_0}^{\omega}=  \bigcup_{x\in
    Z_{n_0}^{\omega}}V_{n_0}(x).
\end{equation*}
Finally let $$V_{n_0}=  \bigcup_{i=1}^{j_{n_0}}V(x_i)$$ and
\begin{equation*}\label{satelite_n0}
{S}_{n_0} = \bigcup_{\omega\in \mathcal P_{n_0}\cup\{\Delta_{0}^{c}\}}{S}_{n_0}^\omega \,\cup\, V_{n_0}.
\end{equation*}
Notice that \( S_{n_0} \) 
$$ H_{n_0}\cap \Delta_0
\subset S_{n_0}\cup\bigcup_{\omega\in\mathcal P_{n_0}}\omega.
$$

\smallskip

\paragraph{\bf General step of induction}
We now proceed inductively and assume that the construction has been carried out up to time \(  n-1  \) for some \(  n> n_{0}  \).
 More precisely,
for each \(  n_{0}\leq k\leq n-1  \) we have a collection of pairwise disjoint sets \(  \mathcal P_{k}=\{ \omega^{x_{1}}_{k},..., \omega^{x_{j_{k}}}_{k}\}  \) which ``return'' at time \(  k+m  \) with \(  0\leq m \leq N  \), and such that for any \(  k\neq k'  \), any two sets \(  \omega\in \mathcal P_{k}\) and \( \omega'\in \mathcal P_{k'}  \)  we have \(  \omega\cap\omega'=\emptyset  \). We also have a set \(  \Delta_{k}  \) which is the set of points which do not yet have an associated return time.
To construct all relevant objects at time \(  n  \), we note first all,
as before,  that there are  
$z_1,\dots, z_{\ell_{n}}\in H_{n}\cap \Delta_{n-1}$
such that  
\[
H_{n}\cap \Delta_{n-1}\subset V_{n}(z_1)\cup\cdots \cup
V_{n}(z_{\ell_{n}}),
\]
and we choose a maximal subset of points
\(  x_1,\dots,x_{j_{n}}\in\{z_1,\dots, z_{\ell_{n}}\}   \)
such that the corresponding sets of type \eqref{D.candidate2}
are pairwise disjoint and contained in \(  \Delta_{n-1}  \). Then we  let
$$
\mathcal P_{n}=\{\omega_{n}^{x_1},\ldots,
\omega_{n}^{x_{j_n}}\}
$$
 These are the elements of the partition $\cp$
constructed in the $n$-step  of the algorithm.
We also define the set of points of \(  \Delta_{0}  \) which do not belong to partition elements constructed up to this point:
\begin{equation*}\label{delta_n}
\Delta_{n} = \Delta_{0}\setminus \bigcup_{\omega\in\mathcal P_{n_0}\cup\cdots\cup \mathcal P_n }\omega.
\end{equation*}
For each   $0\leq i \leq j_n$
we set
\[
R|_{\omega^{x_i}_{n}}=n+m_{i},
\]
where \(  0\leq m_{i} \leq N  \) is the integer associated to \(  \omega^{x_{i}}_{n_{0}}  \) as in \eqref{D.candidate2}.
Let
$$ Z_{n}=\{z_1,\dots, z_{\ell_{n}}\}\setminus\{x_1,\dots, x_{j_{n}}\}
$$
and for any  $\omega\in \mathcal P_{n_0}\cup\cdots\cup \mathcal P_n\cup \{\Delta_{0}^{c}\}$ define
\begin{equation*}\label{auxiliosatelite}
    Z_{n}^{\omega}=\left\{z \in  Z_{n}: \omega^z_{n}\cap \omega\neq\emptyset\right\}
\end{equation*}
and its \emph{$n$-satellite}
\begin{equation*}
    S_{n}^{\omega}=  \bigcup_{z\in
    Z_{n}^{\omega}}V_{n}(z).
\end{equation*}
Finally let  $$V_{n}=  \bigcup_{i=1}^{j_{n}}V(x_i)$$ and
\begin{equation*}\label{satelite_n}
{S}_{n} = \bigcup_{\omega\in \mathcal P_{n_0}\cup\cdots\cup \mathcal P_n\cup \{\Delta_{0}^{c}\}}{S}_{n}^{\omega}\,\cup\, V_n.
\end{equation*}
Note that  for each $n\ge n_0$ one has
\begin{equation}\label{re.saturacao}
H_n\cap \Delta_{n-1}\subset S_n\cup\bigcup_{\omega\in\mathcal P_{n_0}\cup\cdots\cup \mathcal P_n}\omega.
\end{equation}
More specifically we have that \( H_n\cap \Delta_{n-1}\subset S_n \), i.e. all points in \( \Delta_{n-1} \) which have a hyperbolic time at time \( n \) are "covered" by \( S_n \) while the points which have a hyperbolic time at time \( n \) but which are already contained in previously constructed partition elements, are trivially `covered" by the union of these partition elements. The inclusion \eqref{re.saturacao} will be crucial in Section \ref{sec:int} to prove the integrability of the return times. 

%

This inductive construction allows us to define the family
\[
\mathcal P = \bigcup_{n\geq n_{0}}\mathcal P_{n}
\]
of pairwise disjoint subsets of \(  \Delta_{0}  \). At this point there is no guarantee that \(  \mathcal P  \) forms a \(  \leb  \) mod~0 partition of \(  \Delta_{0}  \). This will follow as a corollary of Proposition \ref{d.prop.Sn} below.

\section {Partition on the reference leaf}\label{s.measureHS}

In this section we prove that the elements of \(  \mathcal P  \) defined in the previous section form a \(  \leb_{\Delta}   \) mod 0 partition of the disk \(  \Delta_{0}  \) introduced in \eqref{Delta0}. Some of the partial technical estimates will also be used later to prove the integrability of the return times with respect to \( \leb_{\Delta}  \).

\cle\label{D.estpreball}
There exists $C_{2}>0$ such that
for any \(  n\geq k \geq n_{0}  \) and any \(  \omega\in\mathcal P_{k}  \) we have
\[
\leb_{\Delta} ( S_{n}^{\omega})
  \leq C_2 \leb_{\Delta} \left(\bigcup_{z\in
    Z_{n}^\omega} \omega_{n}^z\right).
\]
  \fle

 \dem

 We note first of all  that, from the construction above, two distinct points \(  z_{1}, z_{2}   \) with the same hyperbolic time \(  n  \) can give rise to the same associated disks \(  \omega^{z_{1}}_{n}=\omega^{z_{2}}_{n}  \). We prove here  that the measure of the union of the hyperbolic pre-disks \(  V_{n}(z)  \) associated to points \(  z\in Z_{n}^{\omega}  \) which give rise to the same disk \(  \omega^{z}_{n}  \) is comparable to the measure of \(  \omega^{z}_{n}  \).
More precisely, we will show that for every
$n\ge 1$  and  $z_1,\dots ,z_N\in H_n$ with
$\omega_{n}^{z_i}=\omega_{n}^{z_1}$ for $1\le i\le N$ we have
\begin{equation}\label{ultimah}
    \leb_\Delta\left(\bigcup_{i=1}^{N} V_n(z_i)\right)\leq C_{2}\leb_{\Delta}(
    \omega^{z_1}_{n}).
   \end{equation}
Notice that \eqref{ultimah} implies the statement in the lemma. Indeed, consider a subdivision of the set  \(  S_{n}^{\omega}  \) of all hyperbolic pre-disks associated to the points in \(  Z_{n}^{\omega}  \) into a finite number of classes such that all hyperbolic pre-disks in each class have the same associated set \(  \omega_{n}^{z}  \). Then apply\eqref{ultimah} to each one. This gives the statement in the lemma.

Thus we just need to prove \eqref{ultimah}.
 For simplicity of notation, for  $1 \leq i\leq
 N$, we write  \(  U_i=V_n(z_i)\) and
\(  B_i=f^{n}(V_{i}) .
\)
We define
$$
X_1= U_1 \qand X_{i}= U_i\setminus\bigcup_{j=1}^{i-1} U_j, \quad \text{for $2\le i\le N$}.
$$
Similarly
$$
Y_1= B_1 \qand Y_{i}= B_i \setminus\bigcup_{j=1}^{i-1} B_j, \quad \text{for $2\le i\le N$}.
$$
 Observe that the \(  X_{i}  \)'s are pairwise disjoint sets whose union coincides with the union of the \(  U_{i}  \)'s, and similarly for the \(  Y_{i}  \)'s and \(  B_{i}  \)'s.
Recalling that $\omega_{n}^{z_i}=\omega_{n}^{z_1}$ for $1\le i\le N$, by the third item of Lemma~\ref{l.contraction2} we have
\begin{equation*}
    \frac{\leb_\Delta (X_i)}{\leb_\Delta (\omega^{z_1}_{n}) }\leq C_1 \frac{\leb_{f^n(\Delta)} (Y_i)}{\leb_{f^n(\Delta)}  f^n(
    \omega^{z_1}_{n})}. 
\end{equation*}
Hence
\begin{eqnarray*}
\frac{\leb_\Delta(U_1 \cup \ldots \cup U_N)}{\leb_\Delta (\omega^{z_1}_{n})}
  &=&
  \frac{\sum_{i=1}^{N}\leb_\Delta (X_i)
   }{\leb_\Delta ( \omega^{z_1}_{n})}\\
  &\leq&
  C_1\frac{\sum_{i=1}^{N}\leb_{f^n(\Delta)} (Y_i)
   }{\leb_{f^n(\Delta)} (f^n( \omega^{z_1}_{n}))} \\
&= &C_1\frac{\leb_{f^n(\Delta)} (B_1\cup \ldots \cup B_N)}{\leb_{f^n(\Delta)} (f^n(
    \omega^{z_1}_{n}))}.
\end{eqnarray*}
We just need to show that the right hand side is bounded above, and for this it is sufficient to show that the denominator
\(  \leb_{f^n(\Delta)}  (f^n(\omega^{z_1}_{n}))  \) on the right hand side is bounded below.  This is clearly true, because by definition of
\(  \omega^{z_1}_{n}  \) we have  \(  m\leq N_{0}  \) such that \(  f^{n+m}(\omega^{z_1}_{n})  \) is a \(  cu  \)-disk of radius \(  \delta_{0}  \).
\cqd

\cre\label{remarco}
The argument used to prove \eqref{ultimah} gives in particular that for each $1\le i\le j_n$ we have
 $\leb_\Delta(V(x_i))\le C_2\leb_\Delta(\omega_n^{x_i}).$
\fre

The next lemma shows that, for each $n$ and
$m$ fixed, the Lebesgue measure on the disk $\Delta$ of the union of sets
$\omega^{z}_{n}$ which intersects an element of partition is
proportional to the Lebesgue measure of that element. The proportion
constant can actually be made uniformly summable in $n$.

 \cle\label{estimativas}There
exists  $C_3> 0$ such that for all $n\geq k\ge n_0$ and $\omega\in
\mathcal{P}_k$ we have
$$
\leb_{\Delta} \left(\bigcup_{z \in
Z_n^\omega}\omega_{n}^z\right)\leq C_3\sigma^{n-k} \leb_{\Delta}(\omega).
$$
\fle

\dem
By construction, given $\omega\in\mathcal P_k$,  there is some hyperbolic pre-disk \(  V_{k}(y)  \) such that
\[
  \omega\subset V_{k}(y) \subset V_{k}^{+}(y)
  \]
and whose images under \(  f^{k}  \) are respectively \(  cu  \)-disks \(  B_{\delta_{1}}^{u}\subset B_{2\delta_{1}}^{u} \) centred  at \(  f^{k}(y)  \). Moreover, there exists some integer \(  0 \leq \ell \leq N_{0}  \) such that \(  f^{k+\ell}(V_{k}(y))  \) \(  u  \)-crosses \(  \mathcal C_0  \) and \(  f^{k+\ell}(\omega)  \) is that part of \(  f^{k+\ell}(V_{k}(y))  \) which projects onto \(  \Delta_{0}  \). Moreover, 
 \(    f^{k} (V_{k}^{+}(y) )\) is a \(  \delta_{1}  \)-neighbourhood of \(    f^{k} (V_{k}(y) )\)
 and so  \(    f^{k+\ell} (V_{k}^{+}(y) )\) contains a  \(  \delta_{1}K_{0}^{-N_{0}}  \)-neighbourhood of \(    f^{k+\ell} (V_{k}(y) )\), where is defined in \eqref{eq:K0}. In particular,
 \begin{equation}\label{eq:conbound}
 \text{\(   f^{k+\ell} (V_{k}^{+}(y) )\) contains a  \(  \delta_{1}K_{0}^{-N_{0}}  \)-neighbourhood of \(    \partial f^{k+\ell} (\omega)\).}
 \end{equation}
 For any \(  n\geq k  \) we let
 $$A^0_{n,k}=\left\{z \in f^{k+\ell}(V_{k}^{+}(y)) : \dist_{f^{k+\ell}(V_{k}^{+}(y)) }(z,\partial f^{k+\ell}(\omega))\leq
    2\delta_0K_0^{N_{0}}\sigma^{n-(k+N_{0})}
    \right\}
    $$
    and
     $$A^1_{n,k}=\left\{z \in f^{k+\ell}(\omega) : \dist_{f^{k+\ell}(V_{k}^{+}(y)) }(z,\partial f^{k+\ell}(\omega))\leq
    2\delta_1K_0^{N_{0}}\sigma^{n-(k+N_{0})}
    \right\}.
    $$
 Observe that \(  A^0_{n,k}  \) and  \(  A^1_{n,k}  \) are both annuli 
  surrounding the boundary   of \(  \omega \) in $f^{k+\ell}(V_{k}^{+}(y)) $, with the particularity that \(  A^1_{n,k}  \)  surrounds only inside $f^{k+\ell}(\omega)$.
  A straightforward calculation gives that
  there is a constant \(  C>0  \), independent of \(  k  \) and \(  n  \), such that
 \begin{equation}\label{eq:Bnk}
\leb_{f^{k+\ell}(V_{k}^{+}(y))}(A_{n,k}^i)\leq C
\sigma^{n-k},\quad i=1,2.
 \end{equation}
Now we see that for $z\in Z_n^\omega$ we have $f^{k+\ell}(\omega_{n}^z)$ contained in $ A^0_{n,k}$ or $ A^1_{n,k}$, depending on the  following two possible cases:
\begin{enumerate}
\item $\omega_n^z\subseteq \omega$.\\
By the first item of Lemma \ref{l.contraction2}  (see also Remark \ref{pretend}),  for each
$\omega_{n}^z$ with \(  z\in Z_{n}^{\omega}  \)  we have
\begin{equation}\label{eq:previous0}
\diam_{f^{k+\ell}(\omega_{n}^z)}( f^{k+\ell}(\omega_{n}^z))\leq 
\diam_{f^{k+\ell}(\omega_{n}^z)}( f^{k+\ell}(V_{n}(z)))
\le 2\delta_1K_0^{N_{0}}\sigma^{n-(k+N_{0})},
\end{equation}
Noting that as $z\notin \omega$ and $\omega_n^z\subseteq \omega$, then $ V_{n}(z)$ necessarily intersects the boundary of $\omega$, and so $f^{k+\ell}(V_{n}(z))$ intersects $\partial f^{k+\ell}(\omega)$. It follows from \eqref{eq:previous0} that
\begin{equation}\label{eq:Bnk12}
f^{k+\ell}(\omega_{n}^z)\subseteq A^1_{n,k}.
 \end{equation}
\item $\omega_n^z\nsubseteq \omega$.
\\
In this case,  \( \omega_{n}^{z}   \) necessarily intersects the boundary of~\(  \omega  \) because \(  z\in Z_{n}^{\omega}  \).  Once more by the first item of Lemma \ref{l.contraction2}  (see also Remark \ref{pretend}),  
we have
\begin{equation}\label{eq:previous}
\diam_{f^{k+\ell}(\omega_{n}^z)}( f^{k+\ell}(\omega_{n}^z))\leq 2\delta_0K_0^{N_{0}}\sigma^{n-(k+N_{0})},
\end{equation}
where we have used the fact that \(  \omega_{n}^{z}  \) is contained  in some hyperbolic pre-disk \(  V_{n}(z)  \) and the term \(  K_{0}^{N_{0}}  \) comes from the fact that \(  \omega_{n}^{z}  \) may require up to a maximum of \(  N_{0}  \) iterates to go from \(  f^{n}(V_{n}(z))  \) to the cylinder \(  \mathcal C_0\), \(  u  \)-crossing it.
Since  
 \( \omega_{n}^{z}   \) intersects the boundary of \(  \omega  \), then  \( f^{k+\ell }(\omega_{n}^{z} )  \) intersects the boundary of \(  f^{k+\ell }( \omega ) \). Recalling that $\sigma<1$ and \eqref{eq.delta_0},  it follows from \eqref{eq:conbound} and \eqref{eq:previous} that
 \begin{equation}\label{eq:Bnk1}
f^{k+\ell}(\omega_{n}^z)\subseteq A^0_{n,k}.
 \end{equation}
\end{enumerate}
%
%
Therefore
\begin{eqnarray*}
\frac{\leb_{V_{k}^{+}(y) }\left(\bigcup_{z \in Z_n^\omega}\omega_{n}^z\right)}{\leb_{V_{k}^{+}(y) } (\omega)}
&\leq&
\widetilde C\,
\frac{\leb_{f^{k+\ell}(V_{k}^{+}(y) )}\left(f^{k+\ell}\left(\bigcup_{z \in Z_n^\omega}\omega_{n}^z\right)\right)}{\leb_{f^{k+\ell}(V_{k}^{+}(y) )} (f^{k+\ell}(\omega))} \\
&\leq&
\widetilde C\,
\frac{\leb_{f^{k+\ell}(V_{k}^{+}(y) )}\left(A^0_{n,k}\right)+\leb_{f^{k+\ell}(V_{k}^{+}(y) )}\left(A^1_{n,k}\right)}{\leb_{f^{k+\ell}(V_{k}^{+}(y) )} (f^{k+\ell}(\omega))},
\end{eqnarray*}
where $\widetilde C>0$ is a uniform constant that incorporates the distortion at the hyperbolic time $k$ given by Lemma~\ref{l.contraction2} and the  distortion of  $f^\ell$   with $\ell\le N_0$.
Recalling that $ f^{k+\ell}(\omega))$ $u$-crosses~$\mathcal C_0$, the result then follows by \eqref{eq:Bnk}, \eqref{eq:Bnk12} and \eqref{eq:Bnk1}.
\cqd

 \cpr \label{d.prop.Sn}
$\displaystyle{\sum_{n=n_0}^{\infty}\leb_{\Delta}({S}_{n}) <\infty}$.
\fpr

\dem
 Observe that
\begin{equation}\label{eq.sumsum}
\sum_{n=n_0}^{\infty}\leb_\Delta({S}_{n}) \le
\sum_{n=n_0}^{\infty}\leb_\Delta\left({S}_{n}^{\Delta_{0}^c }\right)+
\sum_{k=n_0}^\infty\sum_{\omega\in
\mathcal{P}_k}\sum_{n=k}^{\infty}\leb_\Delta({S}_{n}^{\omega})
+\sum_{n=n_0}^\infty \leb_\Delta(V_n).
\end{equation}
We start by estimating  the sum with respect to the satellites of $\Delta_{0}^c$.
Notice that from  Lemma~\ref{l.contraction2} it follows that all hyperbolic pre-disks
\(  V_{n}(x)  \) have diameter \(  \leq 2\delta_{1}\sigma^{n}  \).
Therefore
$$
{S}_{n}^{\Delta_{0}^c}\subset \{x \in \Delta_{0}:\,
\dist(x,\partial\Delta_{0})<2\delta_1\sigma^{n}\},
$$
and so  we can find $\zeta>0$ such that
$$
\leb_\Delta \left({S}_{n}^{\Delta_{0}^c }\right)\leq \zeta\sigma^{n}.
$$
This obviously implies that the part of the sum related to $\Delta_{0}^c$ in~\eqref{eq.sumsum} is finite.

Consider now $n\geq k\ge n_0$. By Lemmas \ref{D.estpreball}  and
\ref{estimativas},
for any   $\omega \in \mathcal{P}_k$ we have
\[
  \leb_\Delta(S_{n}^{\omega})
  \leq C_2C_3 \sigma^{n-k}\leb_\Delta(\omega).
\]
It follows that
\begin{eqnarray*}
 \sum_{k=n_0}^\infty\sum_{\omega\in
\mathcal{P}_k}\sum_{n=k}^\infty\leb_{\Delta}(S_n^{\omega})
&\le&C_{2}C_{3}\sum_{k=n_0}^\infty\sum_{\omega\in
\mathcal{P}_k}\sum_{j=0}^{\infty}\sigma^{j}\leb_{\Delta}(\omega)\\
&=& C_{2}C_{3}\frac{1}{1-\sigma}\sum_{k=n_0}^\infty\sum_{\omega\in
\mathcal{P}_k} \leb_{\Delta}(\omega)\\
&\leq & C_{2}C_{3}\frac{1}{1-\sigma}\leb_{\Delta}(\Delta).
\end{eqnarray*}

Finally, by Remark~\ref{remarco} we have
 $$\sum_{n=n_0}^\infty\leb_\Delta(V_n)\le C_2\sum_{n=n_0}^\infty \sum_{\omega\in \mathcal P_n}\leb_\Delta(\omega)\le C_2\leb_\Delta(\Delta)$$
and this gives the conclusion.
\cqd

We are now ready to show that our inductive construction gives rise to a $\leb_\Delta$ mod 0
partition  of $\Delta_0$. Recall that  $\Delta_0\supset\Delta_{n_0}\supset \Delta_{n_0+1}\supset...$,
where \(  \Delta_{n}  \) is the set of points which does not belong to any element of the collection \(  \mathcal P  \) constructed up to time \(  n  \).
It is enough to show that
\begin{equation}\label{eq:cap}
\leb_\Delta\left(\bigcap_{n}\Delta_{n}\right)=0.
\end{equation}
To prove this, notice that by Proposition \ref{d.prop.Sn}, the sum of the
$\leb_\Delta$ measures of the sets ${S}_{n}$ is finite.  It follows from
Borel-Cantelli Lemma that $\leb_\Delta$ almost every
$x\in \Delta_0$ belongs only to finitely many ${S}_{n}$'s, and therefore one can find $n$ such that
 $x\notin {S}_j$ for $j\ge n$.
 Since  $\leb_\Delta$ almost every  $x\in\Delta_0$ has infinitely many hyperbolic times, it follows from \eqref{re.saturacao}
that  $x\in
\omega$ for some $\omega\in\mathcal P_{n_0}\cup\cdots\cup \mathcal P_n$ and therefore
 \eqref{eq:cap} holds.


\section{The GMY structure}\label{GMY}


We are now ready to define the GMY structure on \(\Omega\) as in the beginning of Section \ref{s.partion}.
Consider the center-unstable disk $\Delta_0\subset \Delta$ as in \eqref{Delta0} and the $\leb_\Delta$ mod 0
partition $\cp$ of $\Delta_0$ defined in
Section~\ref{s.partion}. We define
$$\Gamma^s=\left\{ W^s_{\delta_s}(x):\,x\in \Delta_0\right\}.$$
Moreover, we define
 $\Gamma^u$ as the set of all  local unstable manifolds contained in $\mathcal C_0$ which $u$-cross~$\mathcal C_0$. Clearly,
$\Gamma^u$ is nonempty because $\Delta_0\in \Gamma^u$. We need to see
 that the union of the leaves in $\Gamma^u$ is compact. This follows ideas that we have already used to prove Proposition~\ref{thmAtop}. By the domination
property and Ascoli-Arzelà Theorem, any limit leaf $\gamma_\infty$
of leaves in $\Gamma^u$ is still a $cu$-disk $u$-crossing $\mathcal C_0$. Thus, by definition of $\Gamma^u$, we have
$\gamma_\infty\in\Gamma^u$. We thus define our set \( \Lambda\) with hyperbolic product structure as the intersection of these families of stable and unstable leaves.
The cylinders $\{\cc(\omega)\}_{\omega\in\cp}$ then clearly form a countable collection of
\(s\)-subsets of \( \Lambda\) that play the role of  the sets    $\Lambda_1,\Lambda_2,\dots$ in (P$_1$)
with the corresponding return times \( R(\omega)\).
 It just remains to check that conditions (P$_1$)-(P$_5$) hold.

\subsection{Markov and contraction on stable leaves}
Condition (P$_1$) is essentially an immediate consequence of the construction.  We just need to check that $f^{R(\omega)}(\mathcal C(\omega))$ is a $u$-subset, for any $\omega\in\cp$. Indeed, choosing the integer $n_0$ in the first step of the inductive algorithm sufficiently large, and using the fact that the local stable manifolds are uniformly contracted by forward iterations under  $f$, we can easily see that the ``height'' of $f^{R(\omega)}(C(\omega))$ is at most $\delta_s/4$. Hence, by the choice of $\delta_0$ we have $f^{R(\omega)}(\mathcal C(\omega))$ made by $cu$-unstable disks contained in $\mathcal C_0$. Moreover, as $f^{R(\omega)}(\omega)$ $u$-crosses~$\mathcal C_0$ the same occurs with the local unstable leaves that form $\mathcal C(\omega)$, and so  (P$_1$) holds.
 (P$_2$) is clearly verified under our assumptions.

\subsection{Backward contraction and bounded distortion}
The backward contraction on unstable leaves and bounded distortion, respectively properties (P$_3$)  and  (P$_4$), follow from Lemma~\ref{l.contraction2}. Indeed, by construction, for each $\omega\in\cp$  there is a
hyperbolic pre-ball $V_{n(\omega)}(x)$ containing $\omega$
associated to some point $x\in D$ with $\sigma$-hyperbolic time
$n(\omega)$ satisfying $R(\omega)-N_0\le n(\omega)\le R(\omega)$. It is sufficient to prove the
(P$_3$)  and  (P$_4$) at the time \( n=n(\omega) \) instead of \( R(\omega) \) since the two differ by a finite and uniformly bounded number of iterations whose  contribution to the estimates is also uniformly controlled.

An immediate consequence of \eqref{e.delta1} is that
if   \(y\in K\) satisfies   \( \dist(f^j(x),f^j(y))\leq \delta_1\) for \( 0\leq j \leq n-1\), then
 \( n \) is a \(\sigma^{3/4}\)-hyperbolic time for \( y \), i.e.
$$
\prod_{j=n-k+1}^{n}\left\|Df^{-1}|E^{cu}_{f^j(y)}\right\|\leq\sigma^{3k/4},
\qquad\text{for all $1\le k \le n$.}
$$
Therefore, taking $\delta_s,\delta_0<\delta_1/2$, for any $\gamma\in \Gamma^u$ we have that $n$ is a $\sigma^{3/4}$-hyperbolic
time for every point in $\cc_\omega\cap\gamma$. The backward contraction on unstable leaves and bounded distortion are then consequence of  Lemma~\ref{l.contraction2}, recall Remark~\ref{re.bounded}.

\subsection{Regularity of the foliations}
\label{sec.regularity}
Property (P$_5$)  is standard for uniformly hyperbolic attractors. In the rest of this section we shall adapt classical ideas to our setting.

We begin with the statement of a useful lemma on vector bundles whose proof can be found in \cite[Theorem 6.1]{HP}. Let us recall that a metric $d$ on $E$ is \emph{admissible} if there is a complementary bundle $E'$ over $X$, and an isomorphism $h\colon E\oplus E'\to X\times B$ to a product bundle, where $B$ is a Banach space, such that $d$ is induced from the product metric on $X\times B$.

\cle\label{th.hirsch-pugh}
Let $p\colon E\to X$ be a vector bundle over a metric space $X$ endowed with an admissible metric. Let $D\subset E$ be the unit  ball bundle, and $F\colon D\to D$ a map covering a Lipschitz homeomorphism $f\colon X\to X$. Assume  that there is $0\le \kappa<1$ such that for each $x\in X$ the restriction $F_x\colon D_x\to D_x$ satisfies $\lip(F_x)\le \kappa$. Then
\begin{enumerate}
  \item there is a unique section $\sigma_0\colon X\to D $ whose image is invariant under $F$;
  \item if  $\kappa \lip(f)^\alpha<1$ for some $0<\alpha\le 1$, then $\sigma_0$ is Hölder continuous with exponent~$\alpha$.
\end{enumerate}
\fle

\cpr\label{th.Holder}
Let $f: M\to M$ be a $C^1$ diffeomorphism and $\Omega\subset M$  a compact  invariant  set with a  dominated splitting   $T_\Omega M=E^{cs}\oplus
E^{cu}$. Then the fiber bundles $E^{cs}$ and $E^{cu}$ are Hölder continuous on $\Omega$.
\fpr

\dem We consider only the centre-unstable bundle as the other one is similar.
For each $x\in \Omega$ let $L_x$ be the space of bounded linear maps from $E_x^{cu}$ to $E_x^{cs}$ and let $L_x^1$ denote the unit ball around $0\in L_x$. We define  $\Gamma_x: L_x^1\to L_{f(x)}^1$ as the graph transform induced by
 $Df(x)$:
    $$\Gamma_x(\mu_x)=(Df \vert E^{cs}_x)\cdot\mu_x \cdot (Df^{-1} \vert E^{cu}_{f(x)}).$$
Consider $L$ the vector bundle over $\Omega$ whose fiber over each $x\in \Omega$ is $L_x$, and let $L^1$ be its unit ball bundle. Then $\Gamma : L^1\to L^1$ is a bundle map covering $f\vert \Omega$ with
 $$\lip(\Gamma_x)\le \|Df \mid E^{cs}_x\|
\cdot \|Df^{-1} \mid E^{cu}_{f(x)}\| \le\lambda<1.$$
Let $c$ be a Lipschitz constant for $f|_\Omega$, and choose $0<\alpha\le 1$ small so that  $\lambda c^\alpha<1$. By Lemma~\ref{th.hirsch-pugh} there exists a unique section $\sigma_0\colon M\to L^1 $ whose image is invariant under~$\Gamma$ and it satisfies a Hölder condition of exponent~$\alpha$. This unique section is necessarily the null section.
\cqd

The next result gives precisely (P$_5$)(a).

\cco\label{co.produtorio}
There are $C>0$ and $0<\beta<1$ such that for all $y\in\gamma^s(x)$ and $ n\ge 0$
 $$\displaystyle
 \log \prod_{i=n}^\infty\frac{\det Df^u(f^i(x))}{\det Df^u(f^i(y))}\le C\beta^{n}.
 $$
\fco
\dem As we are assuming that $Df$ is Hölder continuous, it follows from Proposition~\ref{th.Holder} that $\log|\det Df^u|$ is  Hölder continuous. The conclusion is then an immediate consequence of the uniform contraction on stable leaves.
\cqd

To prove (P$_5$)(b) we  introduce some useful notions. We say that $\phi: N\to P$, where $N$ and $P$ are submanifolds of $M$,  is \emph{absolutely continuous} if it is an injective map for which there exists $J:N\to\RR$ such that
 $$
 \leb_P(\phi(A))=\int_A Jd\leb_N.
 $$
$J$ is called the \emph{Jacobian} of $\phi$.
Property (P$_5$)(b) can be restated in the following terms:

\cpr\label{pr.regulstable}
Given
$\gamma,\gamma'\in\Gamma^u$, define
$\phi\colon\gamma'\to\gamma$  by
$\phi(x)=\gamma^s(x)\cap \gamma$. Then $\phi$ is absolutely continuous and  the  Jacobian  of $\phi$ is given by
        $$
        J(x)=
        \prod_{i=0}^\infty\frac{\det Df^u(f^i(x))}{\det
        Df^u(f^i(\phi(x)))}.$$
\fpr
One can easily deduce from Corollary~\ref{co.produtorio} that this infinite product converges uniformly.
The remaining of this section is devoted to the proof of Proposition~\ref{pr.regulstable}.
We start with a general result about the convergence of Jacobians whose proof is given in \cite[Theorem~3.3]{Ma}.

\cle\label{le.jacomane}
Let $N$ and $P$ be manifolds, $P$ with finite volume, and for each $n\ge 1$,  let $\phi_n:N\to P$ be an
absolutely continuous map with Jacobian $J_n$. Assume that
\begin{enumerate}
  \item $\phi_n$ converges uniformly to an injective
continuous map $\phi:N\to P$;
  \item $J_n$ converges uniformly to an integrable function $J: N\to\RR$.
\end{enumerate}
Then
$\phi$ is absolutely continuous with Jacobian $J$.
\fle
For the sake of completeness, we observe that there is a slight difference in our definition of absolute continuity. Contrarily to  \cite{Ma}, and for reasons that will become clear below, we do not impose the continuity of the maps $\phi_n$. However, the proof of \cite[Theorem~3.3]{Ma} uses only the continuity of the limit function $\phi$, and so it still works in our case.

Consider now $\gamma,\gamma'\in\Gamma^u$ and
$\phi\colon\gamma'\to\gamma$  as in Proposition~\ref{pr.regulstable}. The proof of the  next lemma is given in \cite[Lemma 3.4]{Ma} for uniformly hyperbolic diffeomorphisms. Nevertheless, one can easily see that it  is obtained as a consequence of \cite[Lemma 3.8]{Ma} whose proof uses only the existence of a dominated splitting.

\cle\label{le.mane}
For each $n\ge 1$, there is an absolutely continuous  $\pi_n:f^n(\gamma)\to f^n(\gamma')$ with Jacobian $G_n$ satisfying
\begin{enumerate}
  \item $\displaystyle \lim_{n\to\infty}\sup_{x\in \gamma}\left\{ \dist_{f^n(\gamma')}(\pi_n(f^n(x)),f^n(\phi(x))\right\}=0$;
  \item $\displaystyle \lim_{n\to\infty}\sup_{x\in f^n(\gamma)} \left\{ |1-G_n(x)|  \right\}=0$.
\end{enumerate}

\fle

We consider the sequence of  consecutive return times
for points in $\Lambda$,
 \begin{equation*}\label{def.rs}
 r_1=R\qand r_{n+1}=r_{n}+R\circ f^{r_{n}},\quad\text{for $n\ge1$}.
 \end{equation*}
Notice that these return time functions are defined $\leb_\gamma$ almost everywhere on each $\gamma\in \Gamma^u$ and are piecewise constant.

\cre\label{re.sp} Using the sequence of return times one can easily construct a sequence of  $\leb_\gamma$ mod 0  partitions   $(\mathcal Q_n)_n$ by $s$-subsets of $\Lambda$ with $r_n$ constant on each element of $\mathcal Q_n$, for which (P$_1$)-(P$_5$) hold when we take $r_n$ playing the role of $R$ and the elements of $\mathcal Q_n$ playing the role of the $s$-subsets. Moreover, the constants $C>0$ and $0<\beta<1$ can be chosen not depending on $n$.
\fre

We define, for each $n\ge 1$, the map $\phi_n: \gamma\to\gamma'$ as
 \begin{equation}\label{eq.phin}
 \phi_n=f^{-r_n}\pi_{r_n} f^{r_n}.
 \end{equation}
It is straightforward to check that $\phi_n$ is absolutely continuous with   Jacobian
  \begin{equation}\label{eq.Jn}
  J_n(x)=\frac{|\det (Df^{r_n})^u(x)|}{|\det
        (Df^{r_n})^u(\phi_n(x))|}\cdot G_{r_n}(f^{r_n}(x)).
  \end{equation}
Observe that these functions are defined $\leb_\gamma$ almost everywhere. So, we may find a Borel set $A\subset \gamma$ with full $\leb_\gamma$ measure on which they are all defined. We extend $\phi_n$ to $\gamma$ simply by considering $\phi_n(x)=\phi(x)$ and $J_n(x)=J(x)$ for all $n\ge 1$ and $x\in \gamma\setminus A$. Since $A$ has zero $\leb_\gamma$ measure one still has that $J_n$ is the Jacobian of $\phi_n$.

Proposition~\ref{pr.regulstable} is now a consequence of  Lemma~\ref{le.jacomane} together with the next~one.

\cle\label{le.phin}
$(\phi_n)_n$ converges   uniformly to $\phi$ and   $(J_n)_n$ converges uniformly to $J$.
\fle

\dem It is sufficient to prove the convergence of each sequence restricted to $A$ described above. In particular, the expressions of $\phi_n$ and $J_n$ are given by~\eqref{eq.phin} and~\eqref{eq.Jn} respectively.

Let us prove first  the case of $(\phi_n)_n$. Using the backward contraction on unstable leaves given by (P$_3$) and recalling Remark~\ref{re.sp}, we may write for each $x\in\gamma$
\begin{eqnarray*}
  \dist_{\gamma'}(\phi_n(x),\phi(x))  &=& \dist_{\gamma'}(f^{-r_n}\pi_{r_n} f^{r_n}(x),f^{-{r_n}}f^{r_n}\phi(x)) \\
   &\le& C\beta^{r_n}\dist_{f^{r_n}(\gamma')}(\pi_{r_n} f^{r_n}(x),f^{r_n}\phi(x)).
\end{eqnarray*}
Since ${r_n}\to\infty$ as $n\to\infty$ and $\dist_{f^{r_n}(\gamma')}(\pi_{r_n} f^{r_n}(x),f^{r_n}\phi(x))$ is bounded, by Lemma~\ref{le.mane}, we have the uniform convergence of $\phi_n$ to $\phi$.

 Let us prove now that  the case of the Jacobians $(J_n)_n$. By \eqref{eq.Jn}, we have
  $$J_n(x)=\frac{|\det (Df^{{r_n}})^u(x)|}{|\det
        (Df^{{r_n}})^u(\phi (x))|}\cdot \frac{|\det (Df^{{r_n}})^u(\phi (x))|}{|\det
        (Df^{{r_n}})^u(\phi_n(x))|}\cdot G_{{r_n}}(f^{{r_n}}(x)).
        $$
 Using the chain rule and  Corollary~\ref{co.produtorio}, it easily follows that the first term in the product above converges uniformly to $J(x)$. Moreover, by Lemma~\ref{le.mane}, the third term converges uniformly to~1. It remains to see that the middle term also converges uniformly to~1. Recalling Remark~\ref{re.sp}, by  bounded distortion we have
  \begin{eqnarray*}
    \frac{|\det (Df^{{r_n}})^u(\phi (x))|}{|\det
        (Df^{{r_n}})^u(\phi_n(x))|}  &\le &\exp\big(C\dist_{f^{r_n}(\gamma')}(f^{r_n}(\phi(x)),f^{r_n}(\phi_n(x)))^{\eta}\big) \\
     &=& \exp\big(C\dist_{f^{r_n}(\gamma')}(f^{r_n}(\phi(x)),\pi_{r_n}(f^{r_n}(x)))^{\eta}\big).
  \end{eqnarray*}
  Similarly we obtain
  $$\frac{|\det (Df^{{r_n}})^u(\phi (x))|}{|\det
        (Df^{{r_n}})^u(\phi_n(x))|}  \ge  \exp\big(-C\dist_{f^{r_n}(\gamma')}(f^{r_n}(\phi(x)),\pi_{r_n}(f^{r_n}(x)))^{\eta}\big).
        $$
        The conclusion then follows from Lemma~\ref{le.mane}.
 \cqd

\section{Integrability of the return time}\label{sec:int}
In the previous sections we have constructed a GMY structure on \(\Omega\). To complete the proof of Theorem
\ref{thE} it just remains to show that this GMY structure has integrable return times as in \eqref{integrab_return}.
Recall first that the existence of a GMY structure implies the existence of an induced map \( F: \Lambda \to \Lambda\) with an invariant probability measure \(\nu\), see remarks following Theorem \ref{thE}. This measure can be disintegrated into a family of conditional measures on the unstable leaves \( \{\gamma^u\}\) with conditional measures   which are equivalent to Lebesgue measure with densities bounded by uniform constants above and below, see  \cite[Lemma 2]{You98}. We fix one such unstable leaf \( \gamma\in \Gamma^u\) and let \(\bar \nu\) denote the conditional measure associated to \( \nu \) and equivalent to Lebesgue. The integrability of the return times with respect to Lebesgue as in~\eqref{integrab_return} therefore follows immediately from the next result.

\label{intgrreturns}
\begin{Proposition}
The inducing time function $R$ is $\bar{\nu}$-integrable.
\end{Proposition}
\begin{proof}
We first introduce some notation.
For \( x\in \Delta \) we consider the orbit \( x, f(x),..., f^{n-1}(x) \) of the point \( x \) under iteration by \( f \) for some large value of \( n \). In particular \( x \) may undergo several full returns to \( \Delta \) before time \( n \).  Then we define the following quantities:
\begin{align*}
 H^{(n)}(x) &:= \text{number of hyperbolic times for \( x \) before time \( n \)}
 \\
  S^{(n)}(x) &:=\text{number of times \( x \) belongs to a satellite before time \( n \)}
  \\
   R^{(n)}(x)&:= \text{number of returns of \( x \) before time \( n \)}
\end{align*}
Each time that \( x \) has a hyperbolic time, it either then  has a return within some finite and uniformly bounded number of iterations, or by definition it belongs to a satellite. Therefore there exists some constant \( \kappa>0 \) independent of \( x \) and \( n \) such that
\[   R^{(n)}(x)+S^{(n)}(x) \geq \kappa H^{(n)}(x)
 \]
Notice that \( x \) may belong to a satellite or have a return without it having a hyperbolic time itself, since it may belong to a hyperbolic pre-disk of some other point \( y \) which has a hyperbolic time.
Dividing the above equation through by \( n \)  we get
\[
\frac{R^{(n)}(x)}{n}+\frac{S^{(n)}(x)}{n} \geq \frac{\kappa H^{(n)}(x)}{n}
\]
 Recalling that hyperbolic times have uniformly positive asymptotic frequency, there exists a constant \( \theta>0 \) such that \( H^{(n)}(x)/n \geq \theta \) for all \( n \) sufficiently large, and therefore, rearranging the left hand side above gives
 \[
 \frac{R^{(n)}(x)}{n}\left(1+\frac{S^{(n)}(x)}{R^{(n)}(x)}\right) \geq  \kappa\theta > 0
  \]
 Moreover \( S^{(n)}(x)/R^{(n)}(x) \) converges by Birkhoff's ergodic theorem to precisely the average number of times \( \int S d\nu \) that typical points belong to satellites before they return, and from Proposition~\ref{d.prop.Sn} it follows that \( \int S d\nu < \infty \). Therefore,
 we have
 \begin{equation}
 \label{key}   \frac{R^{(n)}(x)}{n} \geq \kappa' > 0
 \end{equation}
 for all sufficiently large \( n \) where \( \kappa'\) can be chosen arbitrarily close to \(\kappa\theta/(1+\int S d\bar{\nu}) \) which is independent of \( x \) and $n$.
 To conclude the proof notice that  \( n/R^{(n)}(x) \) is the average return time over the first \( n \) iterations and thus converges by Birkhoff's ergodic theorem to \( \int \bar{R} d\bar{\nu} \). This holds even if we do not assume a priori that \( \bar{R} \) is integrable since it is a positive function and thus \( \int \bar{R} d\bar{\nu} \) is always well defined and lack of integrability necessarily implies \( \int \bar{R} d\bar{\nu} = + \infty \).
 Thus. arguing by contradiction and assuming that   \( \int \bar{R}d\bar{\nu} = + \infty  \) gives
\( {n}/{R^{(n)}(x)} \to  \int \bar{R}d\bar{\nu} = + \infty \)
and therefore
\( {R^{(n)}(x)}/{n} \to 0\).  This contradicts \eqref{key} and therefore implies that we must have \( \int\bar{ R} d\bar{\nu} < +\infty \)
 as required.
\end{proof}

\section{Liftabilty}\label{lift}

In this section we complete the proof of  Theorem \ref{thF}.
The `if" part of this result is well known and we refer to it in the comments preceeding Theorem \ref{thF}. We therefore just need to show that every SRB measure with positive Lyapunov exponents in the \(E^{cu}\) direction is liftable.  To achieve this,  first of all let \( \Omega\) denote the support of the given SRB measure~\(\mu\). Then \(\Omega\) is invariant under \( f\) and thus under any positive iterate of \( f \).
We will
show in the following proposition that  there exists some \(  N\geq 1  \) such that \(  f^{N}  \) on \( \Omega\) is   nonuniformly expanding, and thus weakly nonuniformly expanding, along \(  E^{cu}  \). We can then apply the conclusions of Theorem \ref{thE} to obtain a GMY structure  for \( f^N\) with integrable return time function \( R \). This easily give a corresponding GMY structure for \(f\) with return time function \(N R \) which is therefore still integrable and therefore,
as explained above,  gives rise to an SRB measure.
By uniqueness of SRB measures it follows that this measure coincides with \( \mu \),  thus proving that \( \mu \) is liftable.

\cpr\label{pr.WNUE}
There exists  $N\ge 1$  such that  \( f^N\) is non-uniformly expanding along \(  E^{cu}  \)
on a set with positive Lebesgue measure.
\fpr

\begin{proof}
We prove first of all that there exists  \( N \geq 1 \) such that
\begin{equation}\label{neg}
\int \log \|(Df^{N}|_{E_{x}^{cu}})^{-1}\| d\mu < 0.
\end{equation}
Indeed,  by assumption all Lyapunov exponents of \(  f  \) along \(  E^{cu}  \) are positive and therefore  all Lyapunov exponents of the map  \(  f^{-1}  \) along \(  E^{cu}  \) are negative. Thus, considering the cocycle \(  (x,v) \mapsto (f^{-1}(x), Df^{-1}(x) v)  \), Oseledets' Theorem implies that there exists \(  \lambda  \) such that
\begin{equation}\label{eq:lift1}
\lim_{n\to\infty}\frac 1n \log \left\|Df^{-1}|_{E^{cu}_{f^{-n+1}(x)}}\cdots Df^{-1}|_{E^{cu}_{x}}\right\| = \lambda<0
\end{equation}
where \(  \lambda  \) is the largest Lyapunov exponent of $f^{-1}$, see \cite[Addendum 4]{Boc}. By the chain rule and the inverse function theorem, we have
\begin{equation}\label{eq:lift2}
Df^{-1}|_{E^{cu}_{f^{-n+1}(x)}}\cdots Df^{-1}|_{E^{cu}_{x}}= \left(Df^{n}|_{E_{f^{-n}(x)}^{cu}}\right)^{-1}.
\end{equation}
Since the sequence
\[
\phi_{n} = \log  \left\| \left(Df^{n}|_{E_{f^{-n}(x)}^{cu}} \right)^{-1} \right\|
\]
satisfies \(  \phi_{n+m}\leq \phi_{n}+\phi_{m}\circ f^{-n} \), using the invariance of \(  \mu  \) with respect to \(  f^{-1}  \) and Kingmann's Subadditive Ergodic Theorem we have, for \(  \mu  \) almost every \(  x  \),
\[
\lim_{n\to \infty}
 \frac 1n \log \left\|\left(Df^{n}|_{E_{f^{-n}(x)}^{cu}}\right)^{-1}\right\|
=  \inf_{n\geq 1} \frac 1n \int \log \left\|\left(Df^{n}|_{E_{f^{-n}(x)}^{cu}}\right)^{-1}\right\| d\mu,
\]
which, together with   \eqref{eq:lift1} and  \eqref{eq:lift2}, gives \eqref{neg}.

 Notice  that \(  \mu  \) may not be ergodic for \(  f^{N}  \), but it can have at most \(  N  \) ergodic components. Indeed, notice first of all
that  any subset $C$ which is  $f^N$-invariant and has
 positive measure, satisfies  $\mu(C)\geq 1/N$:
assume by
contradiction that $\mu(C)< 1/N$ and consider the set
$\cup_{j=0}^{N-1}f^{-j}(C).$
We have that
$$0<\mu\left(\bigcup_{j=0}^{N-1}f^{-j}(C)\right)\leq\sum_{j=0}^{N-1}\mu(f^{-j}(C))<1.$$
This gives a contradiction, because the set is $f$- invariant and
$\mu$ is ergodic.
Now, if $(f^N,\mu)$ is not ergodic, then we  decompose $M$ into
a union of two $f^N$-invariant disjoint sets  with positive measure.
If the restriction of $\mu$ to one of these sets is not ergodic,
then we iterate this process. Note that this must stop after a
finite number of steps with at most $N$ disjoint subsets, since
$f^N$-invariant sets with positive measure have its measure bounded
from below by $1/N$.

Thus, we have that  \(  (f^{N}, \mu)  \) has at most \(  N  \)
ergodic components. By \eqref{neg}, at least one of these ergodic components, whose support
 we denote by \(  \Sigma \),
satisfies
\(
    \int_{\Sigma}\log \|(Df^{N}|{E_{x}^{cu}})^{-1}\|d\mu < 0.
\)
Hence, by Birkhoff's Ergodic Theorem, for $\mu$ almost every $x\in \Sigma$
one has
\begin{equation*}\label{eq.integral}
\lim_{n\rightarrow\infty}\frac{1}{n}\sum_{j=0}^{n-1}
\log\|(Df^N|{E_{f^{Nj}(x)}^{cu}})^{-1}\| =\int_{\Sigma}\log\|(Df^N|{E_{x}^{cu}})^{-1}\|d\mu <0.
 \end{equation*}
This proves that \(  f^{N}  \) is non-uniformly expanding along \(  E^{cu}  \)
 for $\mu$ almost every point in the set~\( \Sigma\). From the assumption that \(  \mu  \) is an SRB measure we have that
 conditional measures of \(  \mu  \) on local unstable manifolds are absolutely continuous with respect to Lebesgue. In particular there is some local unstable manifold \(  \gamma^{u}  \) on which we have non-uniform expansion for  a set of points of  positive \(  \leb_{\gamma^{u}}  \) measure. Considering the union of local stable manifolds through these points and the   absolute continuity of the stable foliation we get the result.
\end{proof}


\begin{thebibliography}{10}



\bibitem{ABV} J. F. Alves, C. Bonatti, M. Viana, {\em
SRB measures for partially hyperbolic systems whose
       central direction is mostly expanding}, Invent. Math. 140 (2000), 351-398.


\bibitem{AlvAraSau}
J.F. Alves, V. Araujo, B. Saussol,
\emph{On the uniform hyperbolicity of some nonuniformly hyperbolic systems},
Proc. AMS
131 (2003),  1303--1309.

\bibitem{AlvAraVas07} 
J. F. Alves, V. Araújo, C. Vásquez, 
\emph{Stochastic stability of non-uniformly hyperbolic diffeomorphisms.} Stoch. Dyn. 7 (2007), no. 3, 299?333.



%




\bibitem{AP} J. F. Alves, V. Pinheiro, {\em
Topological structure of (partially) hyperbolic sets with positive
volume},  Trans. Amer. Math. Soc.  360 , no. 10 (2008), 5551-5569.

\bibitem{AlvPin10} J. F. Alves, V. Pinheiro, 
\emph{Gibbs-Markov structures and limit laws for partially hyperbolic attractors with mostly expanding central direction.} Adv. Math. 223 (2010), no. 5, 1706?1730. 


\bibitem{AviLyuMel03}
A. Avila, M. Lyubich, W de Melo. 2003. \emph{Regular or Stochastic Dynamics in Real Analytic Families of Unimodal Maps.} Inventiones Mathematicae 154 (3): 451-550.

\bibitem{BarPes07}
L. Barreira, Y. Pesin. \emph{Nonuniform Hyperbolicity.} 
 Vol. 115. (2007) Cambridge: Cambridge University Press.


\bibitem{Boc} J. Bochi, {\em The multiplicative ergodic theorem of Oseledets},
http://www.mat.puc-rio.br/$_{\text{\~{}}}$jairo/docs/oseledets.pdf

\bibitem{BDV} C. Bonatti,  L. J. Diaz, M. Viana. {\em Dynamics beyond uniform hyperbolicity:
 A global geometric and probabilistic perspective.} Heidelberg: Springer Verlag, 2004. xviii + 384 p.

\bibitem{BV00}
C.~Bonatti, M.~Viana.
\newblock S{RB} measures for partially hyperbolic systems whose central
  direction is mostly contracting.
\newblock {\em Israel J. Math.}, 115:157--193, 2000.


\bibitem {Bow75} R. Bowen, {\em Equilibrium States and Ergodic Theory of Anosov Diffeomorphisms},
Lectures Notes in Mathematics, volume 470, Springer-Verlag (1975).

\bibitem{BR75}
R.~Bowen, D.~Ruelle.
\newblock The ergodic theory of {A}xiom {A} flows.
\newblock {\em Invent. Math.}, 29:181--202, 1975.



\bibitem{BriPes73}
Brin, M. I.; Pesin, Ja. B. 
\emph{Partially hyperbolic dynamical systems.}
 (Russian) Uspehi Mat. Nauk 28 (1973), no. 3(171), 169?170. 58F15

\bibitem{BurDolPes02}
Burns, K.; Dolgopyat, D.; Pesin, Ya. 
\emph{Partial hyperbolicity, Lyapunov exponents and stable ergodicity.}
 J. Statist. Phys. 108 (2002), no. 5-6, 927?942.


\bibitem{BurDolPesPol08} 
Burns, Keith; Dolgopyat, Dmitry; Pesin, Yakov; Pollicott, Mark 
\emph{Stable ergodicity for partially hyperbolic attractors with negative central exponents.}
 J. Mod. Dyn. 2 (2008), no. 1, 63?81. 


\bibitem{Cas02} A. Castro. 
\newblock Backward inducing and exponential decay of correlations for partially hyperbolic attractors.
\newblock {\em Israel J. Math.} 130 (2002), 29?-75.

\bibitem{CatHeb01} E. Catsigeras, E. Heber.
\newblock SRB measures of certain almost hyperbolic diffeomorphisms with a tangency.
\newblock {\em Discrete Contin. Dynam. Systems} 7 (2001), no. 1, 177?-202



\bibitem{CowWil05} Cowieson, William; Young, Lai-Sang 
\newblock SRB measures as zero-noise limits.
\newblock {\em Ergodic Theory Dynam. Systems} 25 (2005), no. 4, 1115?1138




\bibitem{Dol00} D. Dolgopyat. 
\newblock On dynamics of mostly contracting diffeomorphisms
\newblock {\em Comm. Math. Phys.}, 213:361--201, 2000

\bibitem{Dol04} D. Dolgopyat. 
\newblock On differentiability of SRB states for partially hyperbolic systems
\newblock {\em Invent. Math.} 155 (2004), no. 2, 389?-449.


\bibitem{FieNic04} Field, Michael; Nicol, Matthew 
\newblock Ergodic theory of equivariant diffeomorphisms: Markov partitions and stable ergodicity.
\newblock {\em Mem. Amer. Math. Soc.} 169 (2004), no. 803,

\bibitem{Gou05}S. Gou\"ezel, 
\newblock 
{\em Berry-Esseen theorem and local limit theorem for non-uniformly expanding maps},
 \newblock Ann. Inst. H. Poincar\'e Probab. Statist. 41 (2005) 997-1024.

\bibitem{Hat06} Hatomoto, Jin
\newblock Diffeomorphisms admitting SRB measures and their regularity.
 \newblock Kodai Math. J. 29 (2006), no. 2, 211?226



\bibitem{HP}
M. W. Hirsch, C. C. Pugh, \emph{Stable manifolds and
hyperbolic sets}, 1970 Global Analysis (Proc. Sympos. Pure Math.,
Vol. XIV, Berkeley, Calif., 1968) pp. 133-163 Amer. Math. Soc.,
Providence, R.I.



%

\bibitem{HK} A. Katok, B. Hasselblatt,
\emph{Introduction to the modern theory of dynamical systems},
Encyclopedia of Mathematics and its Applications, 54. Cambridge
University Press, Cambridge, 1995.

\bibitem{KSS07}
O. Kozlovski, W. Shen, S. van Strien,
\emph{Density of hyperbolicity in dimension one.}
 Ann. of Math. (2) 166 (2007), 145-182.



\bibitem{Lyu02}
M. Lyubich, \emph{Almost every real quadratic map is either regular or stochastic}. Ann. of Math. (2) 156 (2002), no. 1, 1-78.




\bibitem{Ma} R. Mañé, \emph{Ergodic theory and diferentiable dynamics},
Springer-Verlag, Berlin, 1987.



\bibitem{MN05}I. Melbourne, M. Nicol, {\em Almost sure invariance principle for non-uniformly hyperbolic systems},
 Comm. Math. Phys. 260 (2005) 131--1456.
\bibitem{MN08}I. Melbourne, M. Nicol, {\em Large deviations for non-uniformly hyperbolic systems},
 Trans. Amer. Math. Soc. 360 (2008) 6661--6676.


\bibitem{P05} J. Palis,  \emph{A global perspective for non-conservative dynamics}, Ann. Inst. H. Poincaré Anal. Non Linéaire 22 (2005), n. 4, 485?507

\bibitem{Pes77}
Ja.~B. Pesin,
\newblock \emph{Characteristic {L}yapunov exponents, and smooth ergodic theory},
\newblock Uspehi Mat. Nauk  \textbf{32} (1977), n.4 (196), 55-112.

\bibitem{Pes04} 
Pesin, Yakov B. 
\emph{Lectures on partial hyperbolicity and stable ergodicity.}
 Zurich Lectures in Advanced Mathematics. European Mathematical Society (EMS), Zürich, 2004


\bibitem{PS82}
Ya. Pesin, Ya. Sinai.
\newblock {\em {G}ibbs measures for partially hyperbolic attractors}.
\newblock {Ergod. Th. \& Dynam. Sys.}, 2:417--438, 1982.

\bibitem{P} V. Pinheiro, {\em Sinai-Ruelle-Bowen measures for weakly
expanding maps }, Nonlinearity  19 (2006), 1185--1200.

\bibitem{Pin11} V. Pinheiro, {\em Expanding measures},
 Ann. Inst. H. Poincar\'e Anal. Non Lin\'eaire 28 (2011), no. 6, 889?939.

\bibitem{HerHerTahUre11}
Rodriguez Hertz, F.; Rodriguez Hertz, M. A.; Tahzibi, A.; Ures, R. 
\emph{Uniqueness of SRB measures for transitive diffeomorphisms on surfaces.}
 Comm. Math. Phys. 306 (2011), no. 1, 35?49. 


\bibitem{RBY08} L. Rey-Bellet, L.-S. Young, \emph{Large deviations in non-uniformly hyperbolic dynamical systems},
Ergodic Theory Dynam. Systems 28 (2008), no. 2, 587-612.

\bibitem{Rue76}
D.~Ruelle,
\newblock \emph{A measure associated with Axiom A attractors},
\newblock  Amer. J. Math.  98 (1976) 619-654.

\bibitem{Sin72}
Ya. G.~Sinai,
\newblock \emph{Gibbs measure in ergodic theory},
\newblock  Russian Math. Surveys  27, n.2 (1972) 21-69.

\bibitem{W} P. Walters, \emph{An introduction to ergodic theory}, Graduate Texts in
 Mathematics 79. Springer-Verlag, New York-Berlin, 1982.

\bibitem{You98} L.-S.Young, {\em Statistical properties of dynamical systems with some hyperbolicity},
Ann. Math., No. 147 (1998), 585 - 650.

\bibitem{You02} L.-S. Young, \emph{What are SRB measures, and which dynamical systems have them?} J. Statist. Phys. 108 (2002), no. 5-6, 733-754.

\bibitem{Vas07} 
Vásquez, Carlos H. 
\emph{Statistical stability for diffeomorphisms with dominated splitting.} Ergodic Theory Dynam. Systems 27 (2007), no. 1, 253?283.

\bibitem{Vas09} 
Vásquez, Carlos H. 
\emph{Stable ergodicity for partially hyperbolic attractors with positive central Lyapunov exponents.}
 J. Mod. Dyn. 3 (2009), no. 2, 233?251. 

\bibitem{ViaYan13} 
Viana, Marcelo; Yang, Jiagang 
\emph{Physical measures and absolute continuity for one-dimensional center direction.}
 Ann. Inst. H. Poincaré Anal. Non Linéaire 30 (2013), no. 5, 845?877.


\end{thebibliography}
\end{document}